\documentclass{amsart}
\usepackage{amsaddr}
\usepackage[margin=1in]{geometry}

\usepackage{mathtools}
\usepackage{marginnote}
\usepackage{amssymb}
\usepackage{algorithmicx}
\usepackage{algpseudocode}
\usepackage{hyperref}
\hypersetup{
	pdftitle={TV Regularization},
	pdfauthor={Paul Manns},
}
\usepackage{multirow}
\usepackage{xcolor}
\usepackage{pgfplots}
\usepackage{tikz}
\usepackage{cite}
\usepackage{doi}
\usepackage{subcaption}
\usepackage{siunitx}
\usepackage[nolist]{acronym}
\usepackage{bbm}
\usepackage{algorithm}
\usepackage{algorithmicx}
\usepackage{todonotes}
\usepackage{booktabs}
\usepackage{multicol}
\usepgfplotslibrary{external}
\usepackage{amsthm}

\newtheorem{theorem}{Theorem}[section]
\newtheorem{proposition}[theorem]{Proposition}
\newtheorem{lemma}[theorem]{Lemma}
\newtheorem{corollary}[theorem]{Corollary}

\theoremstyle{definition}
\newtheorem{definition}[theorem]{Definition}
\newtheorem{assumption}[theorem]{Assumption}
\newtheorem{example}[theorem]{Example}

\theoremstyle{remark}
\newtheorem{remark}[theorem]{Remark}

\numberwithin{equation}{section}

\DeclareMathOperator*{\TV}{TV}
\DeclareMathOperator*{\pred}{pred}
\DeclareMathOperator*{\ared}{ared}

\newcommand*\dd{\mathop{}\!\mathrm{d}}

\newcommand{\F}{\mathcal{F}}

\newcommand{\weakstarto}{\stackrel{\ast}{\rightharpoonup}}

\algdef{SE}[DOWHILE]{Do}{doWhile}{\algorithmicdo}[1]{\algorithmicwhile\ #1}%

\usepackage[nolist]{acronym}

\newcommand{\R}{\mathbb{R}}
\newcommand{\N}{\mathbb{N}}
\newcommand{\Z}{\mathbb{Z}}
\newcommand{\Ha}{\mathcal{H}}

\definecolor{darkgreen}{rgb}{0,0.5,0}

\makeatletter
\long\def\@mn@@@marginnote[#1]#2[#3]{%
  \begingroup
    \ifmmode\mn@strut\let\@tempa\mn@vadjust\else
      \if@inlabel\leavevmode\fi
      \ifhmode\mn@strut\let\@tempa\mn@vadjust\else\let\@tempa\mn@vlap\fi
    \fi
    \@tempa{%
      \vbox to\z@{%
        \vss
        \@mn@margintest
        \if@reversemargin\if@tempswa
            \@tempswafalse
          \else
            \@tempswatrue
        \fi\fi
          \rlap{%
            \if@mn@verbose
              \PackageInfo{marginnote}{xpos seems to be \@mn@currxpos}%
            \fi
            \begingroup
              \ifx\@mn@currxpos\relax\else\ifx\@mn@currxpos\@empty\else
                  \kern-\dimexpr\@mn@currxpos\relax
              \fi\fi
              \ifx\@mn@currpage\relax
                \let\@mn@currpage\@ne
              \fi
              \if@twoside\ifodd\@mn@currpage\relax
                  \kern\oddsidemargin
                \else
                  \kern\evensidemargin
                \fi
              \else
                \kern\oddsidemargin
              \fi
              \kern 1in
            \endgroup
            \kern\marginnotetextwidth\kern\marginparsep
            \vbox to\z@{\kern\marginnotevadjust\kern #3
              \vbox to\z@{%
                \hsize\marginparwidth
                \linewidth\hsize
                \kern-\parskip
                \marginfont\raggedrightmarginnote\strut\hspace{\z@}%
                \ignorespaces#2\endgraf
                \vss}%
              \vss}%
          }%
      }%
    }%
  \endgroup
}
\makeatother

\begin{document}
\title[Sequential Linear Integer Programming for Integer Optimal Control]{Sequential Linear Integer Programming for Integer Optimal Control with Total Variation Regularization}
\author{Sven Leyffer}
\address{Mathematics and Computer Science Division, 
Argonne National Laboratory, Lemont, IL 60439, U.S.A.}
\author{Paul Manns}
\address{Faculty of Mathematics, TU Dortmund University, 44227, Germany}
\date{\today}
\begin{abstract}
We propose a trust-region method that solves a sequence of linear
integer programs to tackle integer optimal control problems
 regularized with a total variation penalty.

The total variation penalty implies that the considered 
integer control problems admit minmizers.
We introduce a local optimality concept for the  problem, which arises from the infinite-dimensional
perspective.
In the case of a one-dimensional domain of the control function,
we prove convergence of the iterates produced
by our algorithm to points that satisfy first-order stationarity
conditions for local optimality.
We demonstrate the theoretical findings on a computational example.
\end{abstract}


\maketitle

\section{Introduction}

Integer optimal control problems cover many practical 
problems from different fields like traffic light 
control \cite{gottlich2017partial}, gas network control 
\cite{hante2017challenges}, and automotive control 
\cite{gerdts2005solving,kirches2013mixed}.
Let $\alpha > 0$, $1 \le d \in \N$, $1 \le M \in \N$,
and $\Omega \subset \R^d$ be a bounded and 
connected Lipschitz extension domain. We consider the 
optimization problem
\begin{gather}\label{eq:p}
\begin{aligned}
\min_{v\in L^{2}(\Omega)}\ & J(v) \coloneqq F(v) + \alpha \TV(v) \\
\text{s.t.}\ & v(x) \in \{\nu_1,\ldots,\nu_M\} \subset \Z
\text{ for almost all (a.a.)\ } x\in \Omega,
\end{aligned}\tag{P}
\end{gather}
where the function $F : L^2(\Omega) \to \R$ is lower
semi-continuous and bounded from below and
$\TV : L^1(\Omega) \to [0,\infty]$ denotes the total variation.
While requiring the inclusion $\{\nu_1,\ldots,\nu_M\} \subset \Z$ improves
the presentation of our arguments, all of our proofs can be adjusted
 such that the claims also hold for $\{\nu_1,\ldots,\nu_M\} \subset \R$.
We assume that the control values $\nu_i$ are in ascending order,
that is, $\nu_i < \nu_{i+1}$ for all $i \in \{1,\ldots,M-1\}$.
Problems of this form arise from optimal control problems by
choosing $F = j \circ S$, where $j$ is some objective and $S$ the control-to-state operator
of some underlying controlled process (e.g.\ some ODE or PDE).

We also introduce the modified control problem, where the
total variation regularizer is absent as in \eqref{eq:phat}:
\begin{gather}\label{eq:phat}
\begin{aligned}
\inf_{v\in L^{2}(\Omega)}\ & F(v) \\
\text{s.t.}\ & v(x) \in \{\nu_1,\ldots,\nu_M\} \subset \Z
\text{ for a.a.\ } x\in \Omega.
\end{aligned}\tag{\^{P}}
\end{gather}

Since the seminal work of Rudin et al.~\cite{rudin1992nonlinear},
regularization by means of total variation has become a
widespread algorithmic tool in mathematical imaging; see, for example,
\cite{vogel1996iterative,dobson1996analysis,chambolle1997image,chan1998total,fornasier2009subspace,bredies2010total,lellmann2014imaging,hintermuller2017optimal}
and also triggered interest in the control community; see, for example,
\cite{loxton2012control,kaya2020optimal,engel2021optimal}.

Recently, Clason et al.\  analyzed and demonstrated the usefulness of
total variation regularization for multi-material topology optimization \cite{clason2018total}. They combined a so-called multibang regularizer
\cite{clason2014multi,clason2018vector}
with a total variation regularizer to promote both that the optimized function
is discrete valued in large parts of the domain and that it has large connected components.
This approach can be considered as solving a tight relaxation of an integer control problem
because the notions of multi-material topology optimization and integer
optimal control coincide if the material function of the former may 
take only discrete values.

Another actively researched technique in the context of integer optimal control is  
combinatorial integral approximation decomposition \cite{sager2011combinatorial}, where the
solution process is split into the solution of a continuous relaxation and the solution of an
approximation problem that can be solved efficiently
\cite{sager2005numerical,sager2012integer,manns2018multidimensional,manns2020improved,kirches2020compactness}.
The major drawback of this technique is
that the fractional-valued solution of the relaxation is approximated in the weak-$^*$ topology
of $L^\infty$, which often yields highly oscillating control functions.
High oscillations hamper meaningful interpretations and implementations of the
computed control functions in practice and correspond to high values of the total variation.
In particular, if a fractional-valued
control $v$ is approximated by discrete-valued ones $v^n$ in the weak-$^*$ topology of $L^\infty$ ($v^n\weakstarto v$), 
the corresponding sequence of total variation terms tends to
infinity; that is, we have $\TV(v^n) \to \infty$ regardless of
the value of $\TV(v)$
 (see, for example,
 the comments in
 \cite[Section 3]{hante2013relaxation}
 and
 \cite[Section 5]{bestehorn2020mixed}).

Therefore, minimization and constraining of total variation terms
have been included in the approximation step of the combinatorial 
integral approximation decomposition in  recent articles
\cite{bestehorn2019switching,sager2019mixed,bestehorn2020mixed}.
Furthermore, the convex relaxation of the multibang regularizer
can be integrated into the combinatorial integral decomposition
approach as well; see \cite{manns2020relaxed}.
Both of these adjustments can considerably reduce the resulting total variation,
or switching costs, 
for a given discretization grid and approximation quality.
High oscillations of the resulting control are often inevitable, however,
if high approximation quality is desired.

The reason is  that the approximation arguments 
underlying the combinatorial integral approximation decomposition
are valid only for the problem class \eqref{eq:phat}
 under the assumptions of \cite{kirches2020compactness} and do not
hold for \eqref{eq:p}. In other words, unlike for the relationship
between \eqref{eq:phat} and its continuous relaxation, the
relationship between \eqref{eq:p} and
its continuous relaxation cannot be exploited with
the combinatorial integral decomposition approach.

In contrast to \eqref{eq:phat}, however, bounded subsets of
the feasible set of \eqref{eq:p} are closed with respect to $\{\nu_1,\ldots,\nu_M\}$-valued 
control functions. To exploit this feature, therefore, in this article we follow a different approach from the
relaxation-based approaches.  
Specifically, we propose to solve \eqref{eq:p}
to local optimality (which will be defined below)
by solving a sequence of linear integer programs (IPs) in which we can
incorporate local gradient information that is available due
to the function space perspective.
Our approach is related to solving a discretized and 
unregularized optimal control problem to global optimality; see
\cite{buchheim2012effective}.
However, the approach in \cite{buchheim2012effective}
does not scale well if very fine discretization meshes are used.
A similar approach to the proposed one was analyzed and
demonstrated recently in \cite{hahn2020binary},
where the authors take the perspective of modifying sets instead
of functions and their algorithm also seeks for a point that satisfies
a local stationarity condition. However, there is no guarantee
that a feasible point that satisfies this condition exists,
and highly oscillating control functions may occur if (a subsequence of) the iterates
converge weakly-$^*$ to an
infeasible (i.e.,  fractional-valued) limit function.
We also mention the articles \cite{exler2007trust,newby2015trust}
in which trust-region-based heuristics and concepts from
nonlinear programming are used to obtain points of low objective
value for finite-dimensional mixed-integer nonlinear programs.

We argue in section \ref{sec:trm_analysis_1d} that our algorithmic approach
has similarities to active set methods.
The switching structure of the resulting control settles down 
eventually, and afterwards only the switch locations
of a fixed sequence of switches are optimized until a 
stationary point is reached. We note that the recently proposed methods on switching time
optimization \cite{de2019mixed,de2020sparse}, which are based on proximal algorithms,
have similarities to this approach.

\subsection{Contributions}
The topological restrictions induced by the 
discreteness  constraint and the $\TV$-regularizer
are strong enough such that \eqref{eq:p} admits optimal solutions.
We exploit this insight to construct a gradient-based
descent algorithm that operates on the feasible set of \eqref{eq:p}.
Specifically, we introduce a trust-region subproblem that becomes an
IP after discretization. We present an algorithm that solves a sequence
of these trust-region subproblems to optimize \eqref{eq:p}.

We restrict to one-dimensional domains
to analyze the asymptotics of the algorithm in function 
space. We show that limit points of the computed
sequence of iterates satisfy a necessary condition for local
optimality of both the trust-region subproblem and \eqref{eq:p}.
The necessary condition may be interpreted
as a first-order condition, and the algorithm therefore
resembles a trust-region algorithm for nonlinear programming.

We provide numerical evidence that validates the theoretical results.
The numerical results also show that---for a fixed fine-discretization grid---the method yields 
results comparable to tackling the discretized problem, a mixed-integer nonlinear program,
with integer programming solvers at much lower computational cost.

\subsection{Notation}
For $p \in [1,\infty)$, $L^p(\Omega)$ denotes the space
of $p$-integrable functions, and $L^\infty(\Omega)$ denotes the
space of all essentially bounded functions.
We let $\N = \{0,1,\ldots\}$ denote the set of natural numbers
including $0$. We denote the Lebesgue measure in $\R^d$ by
$\lambda$. For $k \in \N$, we denote
the $k$-dimensional Hausdorff measure for subsets of $\Omega$
by $\Ha^k$.
A function $u \in L^1(\Omega)$ is of bounded variation if
\[ \TV(u) \coloneqq \sup \left\{
\int_\Omega u(x) \nabla \cdot \phi(x)\,\dd x \,\Big|\,
\phi \in C^1_c(\Omega)^d
\text{ and }
\sup_{s \in \Omega}\|\phi(s)\| \le 1 \right\} < \infty, \]
where $C^1_c(\Omega)$ denotes the class of continuously differentiable
functions that are compactly supported in $\Omega$
and $\|\cdot\|$ denotes Euclidean norm.
The space of functions in $L^1(\Omega)$ with bounded variation is
called $BV(\Omega)$ and is a Banach space when equipped with
the norm $\|u\|_{BV} = \|u\|_{L^1(\Omega)} + \TV(u)$.
We denote the $\{0,1\}$-valued indicator function
of a set $A$ by $\chi_A$.
For an optimization problem (Q), we denote its feasible set
by $\F_{\text{(Q)}}$.

\subsection{Structure of the Remainder of the Paper}
We briefly recall that \eqref{eq:p} has minimizers, and we
introduce our concept of local solutions in
section \ref{sec:existence}.
In section \ref{sec:algorithm} we present our proposed
trust-region algorithm and the trust-region subproblem.
In section \ref{sec:trm_analysis_1d} we derive a suitable notion of stationarity for locally
optimal points and analyze the asymptotics of the trust-region
algorithm, both for the case that $\Omega \subset \R$.
We also show how the trust-region subproblems can be approximated
on uniform grids that discretize $\Omega$.
In section \ref{sec:experiments} we provide computational experiments that
validate our analysis.
We give concluding remarks in section \ref{sec:conclusion}.

\section{Optimal Solutions of $\TV$-Regularized Integer
Control Problems}\label{sec:existence}

We show the existence of optimal controls for \eqref{eq:p} and
introduce the concept of locally optimal solutions for \eqref{eq:p}.

\subsection{Existence of Optimal Solutions}

In contrast to problem \eqref{eq:phat}
\cite{manns2018multidimensional,kirches2020approximation,kirches2020compactness}
or problems with a regularizer induced by a
convex, but not strictly convex, integral superposition operator 
\cite{clason2018vector,hahn2020binary,manns2020relaxed},
we will briefly show that the problem \eqref{eq:p} has
a solution under the following mild assumptions.

\begin{assumption}\label{ass:f_setting}
\mbox{~}
\begin{enumerate}
\item\label{asi:Flsc} Let $F : L^2(\Omega) \to \R$
                      be lower semi-continuous.
\item\label{asi:Fbnd} Let $F$ be bounded from below.
\end{enumerate}
\end{assumption}

Before stating and proving the existence of a minimizer in Proposition
\ref{prp:min_existence}, we briefly state and prove
that the feasible set of \eqref{eq:p}, $\F_{\eqref{eq:p}}$,
is closed with respect to norm convergence in $L^1(\Omega)$,
which will be employed in the derivation of several results
in the remainder of this paper.

\begin{lemma}\label{lem:discreteness_feasibility}
Let $(v^n)_n$ be bounded in $BV(\Omega)$. Let $v^n$
be feasible for \eqref{eq:p} for all $n \in \N$.
Let $v^n \to v$ in $L^1(\Omega)$.
Then $v$ is feasible for \eqref{eq:p}.
\end{lemma}
\begin{proof}
Because $v^n$ converges in $L^1(\Omega)$, it has a pointwise
almost everywhere (a.e.)\ convergent subsequence $(v^{n_k})_k$.
The claim follows from $v^{n_k}(x) \to v(x)$
for a.a.\ $x \in \Omega$
and $v^{n_k}(x) \in \{\nu_1,\ldots,\nu_M\}$ for a.a.\ $x \in\Omega$.
\end{proof}

We combine Lemma \ref{lem:discreteness_feasibility}
with the properties a minimizing sequence of \eqref{eq:p} to prove
the existence of minimizers. A minimizing sequence of an
optimization problem is a sequence of feasible points such that
a corresponding sequence of objective values converges to the infimal
objective value of the optimization problem or is unbounded below.

\begin{proposition}\label{prp:min_existence}
If Assumption \ref{ass:f_setting} holds, then there exists
a minimizer $\bar{v} \in BV(\Omega)$ of \eqref{eq:p}.
\end{proposition}
\begin{proof}
First, we follow the direct method of calculus of variations
to deduce the existence of a convergent minimizing sequence.
Second, we prove that the limit is feasible.

Because $F$ and $\alpha \TV$ are bounded from below,
\eqref{eq:p} is bounded from below as well.
The feasible set is bounded in $L^\infty(\Omega)$
and hence in $L^1(\Omega)$. Therefore, all minimizing sequences
for \eqref{eq:p} in $BV(\Omega)$ are bounded in $L^1(\Omega)$
and, because of the presence of the term $\alpha\TV$, also
in $BV(\Omega)$.
The space $BV(\Omega)$ admits a weak-$^*$ topology,
and we may extract a weakly $^*$-convergent subsequence 
from a minimizing sequence in $BV(\Omega)$ that satisfies
$v^n \to \bar{v}$ in $L^1(\Omega)$ (and by boundedness in $L^2(\Omega)$
and pointwise a.e.\ convergence of a subsequence also in $L^2(\Omega)$).
Thus $F(\bar{v}) \le \liminf_{n\to\infty} F(v^n)$.
Moreover, $\TV(\bar{v}) \le \liminf_{n\to\infty} \TV(v^n)$
follows from the semi-continuity of $\TV$ with respect to
weak-$^*$ convergence in $BV$ or, more generally, because
norms of Banach spaces are always weak-$^*$ lower semi-continuous
if they arise as dual norms of other Banach spaces.

The feasibility of the limit $\bar{v}$
then follows from Lemma \ref{lem:discreteness_feasibility}.
\end{proof}

\begin{remark}
The existence result in Proposition \ref{prp:min_existence}
can be derived from general results on the existence
of minimizers for functionals that are regularized by
total variation; see, for example,
\cite[Chap.\,4\,\&\,5]{ambrosio2000functions}.
The case of binary controls is mentioned explicitly
in \cite[Cor.\ 2.6]{burger2012exact}.
Because
the argument can be stated  briefly and shows how the
boundedness of the $\TV$-term implies the existence
of discrete-valued minimizers, we provide an explicit proof here.
\end{remark}

\begin{remark}
Another mode of convergence in $BV(\Omega)$ that lies
between weak-$^*$ and norm convergence is strict convergence.
This means that in additional to $v^{n_k} \to v$ in $L^1(\Omega)$,
it holds that $\TV(v^{n_k}) \to \TV(v)$.
Constructing strictly convergent sequences
to local minimizers
will be guiding us in our algorithm construction below.
\end{remark}

\begin{remark}
The proof uses a weak-$^*$ convergent minimizing sequence.
We cannot assume uniform convergence or strict convergence for
general minimizing sequences. For example, the
typewriter sequence\footnote{see \url{https://terrytao.wordpress.com/2010/10/02/245a-notes-4-modes-of-convergence/}} is a $\{0,1\}$-valued weak-$^*$-convergent
sequence in $BV(\Omega)$ with limit $0$ but $\TV$-value
of $2$ for all iterates. It is also not pointwise convergent.
\end{remark}

\begin{remark}
As is used in the proof of Lemma \ref{lem:discreteness_feasibility},
we can always obtain $v^n \to v$ in $L^p(\Omega)$ for all $p \in [0,\infty)$
for feasible controls $v^n$ by virtue of Lebesgue's dominated
convergence theorem because
$\|v^n\|_{L^\infty(\Omega)} \le \max\{ |\nu_i| \,|\, i\in\{1,\ldots,m\}\}$ 
and $\Omega$ is a bounded domain. This allows us to only assume lower
semicontinuity instead of the usual weak lower semicontinuity of $F$.
If additional fractional-valued controls are considered that are
not regularized with a $\TV$-penalty but, for example, with an $L^2$-penalty,
$F$ has to be weakly lower semicontinuous for the analysis
to hold.
\end{remark}

\subsection{Locally Optimal Solutions}
We start with the definition of our concept of locally optimal solutions.
\begin{definition}\label{dfn:roptimal}
Let $v$ be feasible for \eqref{eq:p} and $r > 0$.
Then, we say that $v$ is $r$-optimal if and only if
\[ F(v) + \alpha \TV(v) \le F(\tilde{v}) + \alpha \TV(\tilde{v})
   \text{ for all } \tilde{v} \text{ feasible for } \eqref{eq:p}
   \text{ with } \|v - \tilde{v}\|_{L^1(\Omega)} \le r. \]
\end{definition}
If $v$ is optimal for \eqref{eq:p}, then
$v$ is also $r$-optimal for all $r > 0$. Thus, $r$-optimality
is a necessary condition for optimality.
This optimality concept is similar to the notion of \emph{local minimizers}
for finite-dimensional mixed-integer nonlinear programs \cite{newby2015trust},
where optimality of the integer solution in a neighborhood 
in $\R^n$ is considered.

In a Euclidean space, one can always choose a neighborhood around a feasible integer point
that is small enough such that it does not contain any further 
feasible integer points.
This is not true in the infinite-dimensional case.
On the contrary, the situation in the following example is
generic.
\begin{example}
We consider the case of an open domain
$\emptyset \neq \Omega \subset \R^d$, $M \ge 2$
and $v \in BV(\Omega) \cap \F_{\eqref{eq:p}}$.
We may consider a ball $B$ and
construct $\hat{v} \in \F_{\eqref{eq:p}}$ by setting
\[ \hat{v}(x) \coloneqq \left\{
\begin{aligned}
\nu_{i + 1} & \text{ if } x \in B \text{ and } v(x) = \nu_i
\text{ for } i < M, \\
\nu_{M - 1} & \text{ if } x \in B \text{ and } v(x) = \nu_M, \\
v(x) & \text{ if } x \in \Omega\setminus B.
\end{aligned}
\right.
\]
Then $\|v - \hat{v}\|_{L^1(\Omega)} \le \max_{i\neq j} |\nu_i - \nu_j| \lambda(B)$, which tends to zero
when driving the radius of $B$ to zero.
\end{example}

\begin{remark}
Optimizing in $L^1$-neighborhoods and considering the total variation in the objective
rather than as a constraint implies that we base our approach in the weak-$^*$ topology of the space $BV(\Omega)$.
Alternatively, we could have considered the norm topology (with the additional restriction $\TV(\tilde{v} - v) \le r$)
or the strict topology (with the additional restriction $|\TV(\tilde{v}) - \TV(v)| \le r$) on $BV(\Omega)$.

We believe that the proposed approach is beneficial for several reasons. The norm and strict topology are more restrictive.
In particular, for the setting of one-dimensional domains with finitely many integer-valued controls
$\TV(\tilde{v} - v) \le r$ implies that $\tilde{v} - v$ is a constant function for $r < 1$,
see also \eqref{eq:tv_diff_inz_1d} below.
Moreover, as we will see in the convergence proofs of the proposed algorithmic framework,
the regularization with $\alpha \TV(v)$ (which may be interpreted as a \emph{soft constraint})
is sufficient to obtain strict convergence of the produced subsequences without enforcing the strict topology
in the constraint of the trust-region subproblem. Therefore, our algorithmic framework allows that
nontrivial changes in the switching structure in each subproblem are feasible but the obtained
strict convergence implies that the switching structure eventually settles over the iterations,
see also our Remark \ref{rem:active_set_idea} after the convergence proofs.
\end{remark}

\section{Sequential Linear Integer Programming Algorithm}\label{sec:algorithm}

In this section we develop a function space algorithm
to compute $r$-optimal points of \eqref{eq:p}.
We use a trust-region strategy for globalization.
The trust-region subproblem and basic
properties are provided in section \ref{sec:trsub}, and the
algorithm is presented in section \ref{sec:slip}.
We show that, after discretization of the control with
piecewise constant functions (discontinuous Galerkin
elements of order 0), the trust-region subproblem
becomes an IP.

\subsection{Trust-Region Subproblem}\label{sec:trsub}

The concept of $r$-optimality as introduced in Definition
\ref{dfn:roptimal} is local in terms of changes of a function
$v$ with respect to $\|\cdot\|_{L^1(\Omega)}$.
We can find a sufficient condition for an $r$-optimal point by
studying the following trust-region subproblem, where the
objective is linearized partially around $\tilde{v}$; that is,
we linearize $F$ while using the exact $\TV$ term: 
\begin{gather}\label{eq:tr}
(\text{TR}(\tilde{v},\tilde{g},\Delta)) \coloneqq
\left\{
\begin{aligned}
\min_{v\in L^{2}(\Omega)}\ &  (\tilde{g}, v - \tilde{v})_{L^2(\Omega)}
           + \alpha \TV(v)-\alpha \TV(\tilde{v}) \\
\text{s.t.}\ & \|v - \tilde{v}\|_{L^1(\Omega)} \le \Delta,\\
& v(x) \in \{\nu_1,\ldots,\nu_M\} \text{ a.e.}
\end{aligned}\tag{TR}
\right.
\end{gather}
In \eqref{eq:tr}, we use $\tilde{g} = \nabla F(\tilde{v})$ if $F$
is differentiable at $v \in \F_{\eqref{eq:p}}$.
Otherwise, we assume that
$\tilde{g} \in \partial F(\tilde{v}) = \{ g \in L^2(\Omega) \,|\,
F(v) - F(\tilde{v}) \ge (g, v - \tilde{v})_{L^2(\Omega)} \}$,
and in particular $\partial F(\tilde{v}) \neq \emptyset$.

\begin{remark}
We highlight that for the intended use in optimal control
a common structure is $F = j \circ S$, where $j$ is some objective 
functional, for example, of tracking-type, and $S$ is the solution operator
of some ODE or PDE. Then $\tilde{g} = \nabla F(\tilde{v})$ can
be determined as usual with one solve of the adjoint equation.
\end{remark}

\begin{proposition}\label{prp:tr_well_defined}
Let $\tilde{v} \in \F_{\eqref{eq:p}} \cap BV(\Omega)$,
$\tilde{g} \in L^2(\Omega)$, and $\Delta > 0$.
Then, $(\text{\emph{\ref{eq:tr}}}(\tilde{v},\tilde{g},\Delta))$ has a minimizer.
\end{proposition}
\begin{proof}
It holds that
$\tilde{v} \in \F_{(\text{\ref{eq:tr}}(\tilde{v},\tilde{g},\Delta))}$,
and thus the feasible set is nonempty. The constraints imply
boundedness of the feasible set in $L^1(\Omega)$ and
$L^\infty(\Omega)$. Together with the $\TV(v)$-term, the objective
is bounded below.

We consider a minimizing sequence of $(\text{\ref{eq:tr}}(\tilde{v},\tilde{g},\Delta))$; that is,
$(v^n)_n \subset \F_{(\text{\ref{eq:tr}}(\tilde{v},\tilde{g},\Delta))}$
with corresponding objective values converging to the 
infimal value of $(\text{\ref{eq:tr}}(\tilde{v},\tilde{g},\Delta))$. The $\TV(v)$-term implies boundedness of the sequence $(v^n)_n$
in $BV(\Omega)$ and in turn
weak-$^*$ convergence $v^{n_k} \weakstarto v^*$ in $BV(\Omega)$
of a subsequence $(v^{n_k})_k$ for some limit $v^* \in BV(\Omega)$. 
We use $v^{n_k} \to v^*$ in $L^1(\Omega)$ for a subsequence denoted
by the same symbol and Lemma
\ref{lem:discreteness_feasibility} to deduce that
$v^* \in \F_{(\text{\ref{eq:tr}}(\tilde{v},\tilde{g},\Delta))}$.
With the argument from Proposition \ref{prp:min_existence}
we obtain $v^{n_k} \to v^*$ in $L^2(\Omega)$ as well.
The continuity of $(g, \cdot - \tilde{v})_{L^2(\Omega)}$
and the lower-semicontinuity of $\TV$ imply that
$v^*$ is a minimizer.
\end{proof}

\begin{proposition}\label{prp:tr_convex_sufficient}
Let $F : L^2(\Omega) \to \R$ be convex.
Then, $\tilde{v} \in \F_{\eqref{eq:p}}$
is $r$-optimal for \eqref{eq:p} if the optimal
value of $(\text{\emph{\ref{eq:tr}}}(\tilde{v},\tilde{g},\Delta))$
with $\Delta = r$ and some $\tilde{g} \in \partial F(\tilde{v})$
is nonnegative.
\end{proposition}
\begin{proof}
Let the optimal value of $(\text{\ref{eq:tr}}(\tilde{v},\tilde{g},\Delta))$
be nonnegative for $\tilde{g} \in \partial F(v)$ and $\Delta > 0$.
Then it follows that
\[ 0 \le (\tilde{g}, v - \tilde{v})_{L^2(\Omega)}
+ \alpha \TV(v) - \alpha \TV(\tilde{v})\]
for all feasible $v \in \F_{(\text{\ref{eq:tr}}(\tilde{v},\tilde{g},\Delta))}$.
The convexity of $F$ yields
\[ (\tilde{g}, v - \tilde{v})_{L^2(\Omega)} \le F(v) - F(\tilde{v}). \]
Combining both inequalities yields the claim.
\end{proof}

Proposition \ref{prp:tr_convex_sufficient} shows that
if $F$ is convex, then the nonnegativity of the optimal objective
value of \eqref{eq:tr} is a sufficient condition for $r$-optimality.
However, deriving a necessary condition for stationarity is more
involved and will be discussed in section \ref{sec:trm_analysis_1d}.

Next, we introduce an algorithm that solves subproblems of the
form \eqref{eq:tr} to find an optimal solution of \eqref{eq:p}.

\subsection{Algorithm Statement}\label{sec:slip}

Our algorithmic approach for solving \eqref{eq:p} to
$r$-optimality for some $r > 0$ is formalized in Algorithm
\ref{alg:trm}.
The algorithm consists of an outer and an inner loop.
In each outer iteration, the trust-region radius is reset,
and then the inner loop is executed to compute the next iterate.

The inner loop solves the trust-region subproblem \eqref{eq:tr}
for a sequence of shrinking trust-region radii until the
predicted reduction, the negative objective of \eqref{eq:tr},
is less than or equal to zero or a sufficient decrease condition
is met.
If the predicted reduction is less than or equal to zero, 
Proposition \ref{prp:tr_convex_sufficient} implies that the
current iterate $v^{n-1}$ is $r$-optimal with
$r = \Delta^{n,k}$ if $F$ is convex.
Consequently, Algorithm \ref{alg:trm} terminates in this case.

A usual candidate for a sufficient decrease condition
for trust-region problems is the inequality
\begin{gather}\label{eq:1d_accept}
\ared(v^{n-1},\bar{v}^{n,k}) \ge \sigma \pred(v^{n-1},\Delta^{n,k})
\end{gather}
for some $\sigma \in (0,1)$, where
\[ \ared(v^{n-1},\bar{v}^{n,k}) =
F(v^{n-1}) + \alpha \TV(v^{n-1})
- F(\bar{v}^{n,k}) - \alpha\TV(\bar{v}^{n,k})
\]
denotes the actual reduction achieved by the candidate $\bar{v}^{n,k}$
computed by the trust-region subproblem and the predicted reduction
\[ \pred(v^{n-1},\Delta^{n,k})
= (\tilde{g}^{n-1},v^{n-1} - \bar{v}^{n,k})_{L^2(\Omega)}
  + \alpha \TV(v^{n-1}) - \alpha\TV(\bar{v}^{n,k})
\]
is the negative objective of
$(\text{\ref{eq:tr}}(v^{n-1},\tilde{g}^{n-1},\Delta^{n,k}))$.
We show in section \ref{sec:trm_analysis_1d}
for the case $d = 1$ either that the condition \eqref{eq:1d_accept} implies  finite
termination of the inner loop or that $v^{n-1}$ satisfies a necessary optimality
condition. If the inner loop always terminates finitely, the sequence of iterates
has weak-$^*$ accumulation points that all satisfy the necessary optimality condition.

\begin{algorithm}[bht]
\caption{Sequential linear integer programming method (SLIP) to seek $r$-optimal points of \eqref{eq:p}}\label{alg:trm}
\begin{flushleft}
\textbf{Input:} $F$ differentiable, and satisfying
Assumption \ref{ass:f_setting}.

\textbf{Input:} $\Delta^0 > 0$, $v^0 \in \F_{\eqref{eq:p}}$, $\sigma \in (0,1)$.
\end{flushleft}

\begin{algorithmic}[1]
	\For{$n = 1,\ldots$}
	\State $k \gets 0$
	\State $\Delta^{n,0} \gets \Delta^0$
	\Repeat
		\State\label{ln:gtilde} $\tilde{g}^{n-1} \gets$ 
		choose element of $\partial F(v^{n-1})$.
		\State\label{ln:trstep} $\bar{v}^{n,k} \gets$ minimizer of
		$(\text{TR}(v^{n-1},\tilde{g}^{n-1},\Delta^{n,k}))$.
		\State\label{ln:pred} $\pred(v^{n-1},\Delta^{n,k}) \gets
		(\tilde{g}^{n-1}, v^{n-1} - \bar{v}^{n,k})_{L^2(\Omega)}
		         + \alpha \TV(v^{n-1}) - \alpha \TV(\bar{v}^{n,k})$
		\State $\ared(v^{n-1},\tilde{v}^{n,k}) \gets
		         F(v^{n-1}) + \alpha \TV(v^{n-1})
		         - F(\bar{v}^{n,k}) - \alpha\TV(\bar{v}^{n,k})$	
		\If{$\pred(v^{n-1},\Delta^{n,k}) \le 0$}
			\State Terminate. The predicted reduction for $v^{n-1}$
			is zero ($v^{n-1}$ is $\Delta^{n,k}$-optimal if $F$ is convex).
		\ElsIf{$\ared(v^{n-1},\bar{v}^{n,k}) < \sigma \pred(v^{n-1},\Delta^{n,k})$}\label{ln:suffdec}
			\State $k \gets k + 1$		
			\State $\Delta^{n,k} \gets \Delta^{n,k-1} / 2$.
		\Else
			\State $v^n \gets \bar{v}^{n,k}$
			\State $k \gets k + 1$
		\EndIf
	\Until{$\ared(v^{n-1},\bar{v}^{n,k-1}) \ge \sigma \pred(v^{n-1},\Delta^{n,k-1})$
	// sufficient decrease achieved}
	\EndFor
\end{algorithmic}
\end{algorithm}

\begin{remark}
In this work we analyze a function
space algorithm. Specifically,
Line \ref{ln:trstep} cannot
be implemented exactly.
To obtain our numerical results,
we solve a sequence of
finite-dimensional 
subproblems that approximate
the infinite-dimensional subproblems.
\end{remark}

\subsection{Subproblems as Linear Integer Programs for DG0 Control Discretizations}\label{sec:trip}

The feasible control functions can take only
finitely many values and we may restrict our
considerations to a piecewise constant
representative for any $v \in \F_{\eqref{eq:p}}$
because of the choice $\alpha > 0$.
Let $\mathcal{T}$ be a partition of $\Omega$
into finitely many
polytopes of dimension $d$ with the set of interior facets
$\mathcal{E} \subset \mathcal{T}\times \mathcal{T}$.
Let $v \in \F_{\eqref{eq:p}}$ and
$v(x) = \sum_{T\in\mathcal{T}} v_T \chi_{T}(x)$
for a.a.\ $x \in \Omega$ with
$v_T \in \{\nu_1,\ldots,\nu_M\}$, 
where $\chi_T(x)$ is the characteristic function of
$T \in \mathcal{T}$;
that is, $v$ takes the value $v_T$ on grid cell $T \in \mathcal{T}$.
Let $\ell_E = \Ha^{d-1}(E)$
denote the $d-1$-dimensional Hausdorff measure
of facet $E \in \mathcal{E}$. 
We define the selector functions for the grid cells
that are connected by facet $E$ as $T_1$ and $T_2$; that
is, $E = (T_1(E), T_2(E))$.
Moreover, let $[v]_E$ denote the jump height of the function
$v$ across the facet $E$; that is, $[v]_E = v_{T_1(E)} - v_{T_2(E)}$.

With this terminology, we can write the quantities required in \eqref{eq:tr}
as
\begin{align*}
\TV(v) &=
   \sum_{E \in \mathcal{E}} \int_E \left|[v]\right|\dd \Ha^{d-1}
   = \sum_{E \in\mathcal{E}} \ell_E \left|[v]_{E}\right|
   = \sum_{E \in\mathcal{E}} \ell_E \left|v_{T_1(E)} - v_{T_2(E)}\right|,\\
(\tilde{g}, v - \tilde{v})_{L^2(\Omega)}
&= \sum_{T \in \mathcal{T}}(v_T - \tilde{v}_T) \int_{T} \tilde{g}(x) \dd x,\\
\|v - \tilde{v}\|_{L^1(\Omega)} &= \sum_{T\in\mathcal{T}} |v_T - \tilde{v}_T| \lambda(T),
\end{align*}
where $\lambda$ denotes the
Lebesgue measure on $\R^d$
and $\tilde{g} = \nabla F(\tilde{x}_T)$
in the differentiable case.
We use these
expressions and auxiliary variables to transform \eqref{eq:tr}
with the ansatz
$v = \sum_{T\in\mathcal{T}} v_T \chi_{T}$
for a fixed partition $\mathcal{T}$
into the IP

\begin{gather}\label{eq:trip}
(\text{TRIP}(\tilde{v},\tilde{g},\Delta))
\coloneqq
\left\{
\begin{aligned}
\min_{v_T,u_T,w_E,\omega}\ &
     \sum_{T \in \mathcal{T}} c_T (v_T - \tilde{v}_T)
     + \alpha \omega - \alpha \TV(\tilde{v}) \\
\text{s.t.}\ & -u_T \le v_T - \tilde{v}_T\le u_T \text{ for all }
               T\in\mathcal{T}, \\
             & u_T \ge 0 \text{ for all } T \in\mathcal{T},\\
             & {\sum_{T \in \mathcal{T}}} u_T\lambda(T)
               \le \Delta, \\
			 & -w_E \le \tilde{v}_{T_1(E)} - \tilde{v}_{T_2(E)} \le w_E
			   \text{ for all } E \in \mathcal{E}, \\
			 & w_E \ge 0 \text{ for all } E \in \mathcal{E}, \\
			 & {\sum_{E \in \mathcal{E}}} w_E \ell_E \le \omega, \\
             & \tilde{v}_T \in \{\nu_1,\ldots,\nu_M\}
               \text{ for all } T \in\mathcal{T}.
\end{aligned}
\right.
\tag{TRIP}
\end{gather}

Here, the following real-valued scalar
quantities are fixed inputs for the problem
$(\text{TRIP}(v,\tilde{g},\Delta))$.
\begin{enumerate}
\item $\tilde{v}_T \in \{\nu_1,\ldots,\nu_M\}$ for all $T \in \mathcal{T}$.
\item $c_T \coloneqq \int_T \tilde{g}(x)\,\dd x \in \R$
      for all $T \in \mathcal{T}$.
\item $\lambda(T) \ge 0$ for all $T \in \mathcal{T}$.
\item $\alpha > 0$.
\item $\TV(\tilde{v}) \ge 0$.
\item $\Delta \ge 0$.
\item $\ell_E \ge 0$ for all $E \in \mathcal{E}$.
\end{enumerate}

\begin{remark}
Generally, one has no need to solve the 
$(\text{{TRIP}}(\tilde{v},\tilde{g},\Delta))$ exactly.
If the upper bound (incumbent solution) of an iterate in
the IP satisfies the acceptance criterion of the trust-region step,
then the solution algorithm for \eqref{eq:trip} may terminate early.
The efficacy of such a
strategy, however, depends on the costs
to evaluate $F$ or a surrogate with
sufficient accuracy.
\end{remark}

Every implementation of \eqref{eq:trip} requires an additional
approximation step; specifically, the objective coefficient
$c_T = \int_T \tilde{g}(x)\dd x$ requires  numerical
quadrature.
In the case of a uniform grid, where also $\ell_E = \ell_F$ holds for
all $E$, $F \in \mathcal{E}$ (for example, a grid of squares),
the order of the optimality error of the solution of
\eqref{eq:trip} to the solution of (TRIP$^h$), where
the coefficients $c_T$ are replaced by coefficients $c_T^h$
obtained with  quadrature, is of the same order as the quadrature
error.

To see this, let $c_T^h \approx \int_T \tilde{g}(x)\,\dd x$
be such that
\[ c_T (1 - \varepsilon) \le c_T^h \le c_T (1 + \varepsilon) \]
for all grid cells $T \in \mathcal{T}$.
Let $v_*$ denote the solution of \eqref{eq:trip}, and
let $v_*^h$ denote the solution of (TRIP$^h$).
Because $\ell_E$ is constant, the second term of the objective is an integer
multiple of $\alpha\ell_E$.
Therefore, it follows
that $\alpha \omega_*^h = \TV(v^h_*) = \TV(v_*) = \alpha \omega_*$ if $\varepsilon < \frac{\alpha\ell_E}{|\mathcal{T}|\max_{i,j}|\nu_i - \nu_j|}$.
The terms $\sum_T c_T\tilde{v}_T$ as well as $-\alpha \TV(\tilde{v})$ are constant and may be disregarded. Therefore, the objective difference between
the optimal solution with and without discretization of $c_T$ is 
\[ e_h \coloneqq \left|\sum_{T\in\mathcal{T}} c_Tv_{*,T} - \sum_{T\in\mathcal{T}} c_T^h v_{*,T}^h\right|  \]
if $\varepsilon < \frac{\alpha\ell_E}{|\mathcal{T}|\max_{i,j}|\nu_i - \nu_j|}$.
Using the optimality for both problems and the fact that
\[ \left|\sum_{T\in\mathcal{T}}(c_T^h - c_T) v_T\right| \le |\mathcal{T}|\varepsilon \]
implies the bound
\[ e_h \le |\mathcal{T}|\varepsilon \]
on the objective error of \eqref{eq:trip} in this case.
This means that if we want to obtain an objective
error $e_h$ for \eqref{eq:trip},  we need to enforce
an approximation error of $e_h / |\mathcal{T}|$
for the coefficients $c_T^h$.

\begin{remark}
After discretization of the quantities \eqref{eq:trip}
can be solved general purpose IP solvers. Efficient combinatorial
algorithms that exploit the structure of \eqref{eq:trip} are
preferable and subject to ongoing research.
\end{remark}

\section{Analysis of Algorithm \ref{alg:trm} for $d = 1$}\label{sec:trm_analysis_1d}

We analyze Algorithm \ref{alg:trm} in function space for $d = 1$.
Our analysis builds on properties of the $\TV$-term and geometric
constructions that do not generalize easily for $d \ge 2$. We restrict
our analysis to $d = 1$ in this work and comment on arguments that do not
work for $d \ge 2$ in the respective subsections (see, e.g.,\ 
Remarks \ref{rem:dge2_construction_fail} and \ref{rem:step_accept}).

In section \ref{sec:bv} we provide results on functions of bounded 
variation, which we use repeatedly in the remainder.
Then, we derive first-order necessary conditions
for $r$-optimal points in section \ref{sec:first_order},
followed by a sufficient decrease condition and asymptotics
of the inner loop of Algorithm \ref{alg:trm} in
section \ref{sec:suff_dec}. We prove convergence of Algorithm 
\ref{alg:trm} in function space in section \ref{sec:asymptotics}.
Moreover, in section \ref{sec:tr_trip_approx} we show how \eqref{eq:trip} can approximate \eqref{eq:tr}
on uniform grids.

Our analysis requires the assumption below.
\begin{assumption}\label{ass:d1}
Let $F : L^2(\Omega) \to \R$
be twice continuously
Fr\'{e}chet differentiable such that for all $\xi \in L^2(\Omega)$,
the bilinear form induced by the Hessian
\[ \nabla^2 F(\xi) : L^2(\Omega)\times L^2(\Omega) \to \R \]
satisfies
\[ |\nabla^2 F(\xi)(u,w)| \le C\|u\|_{L^1(\Omega)}\|w\|_{L^1(\Omega)} \]
for some $C > 0$ and all $u$, $w \in L^2(\Omega)$.
\end{assumption}
Assumption \ref{ass:d1} ensures that $\nabla F(v) \in L^2(\Omega)$
for all feasible $v$ and therefore Proposition \ref{prp:tr_well_defined}.
Assumption \ref{ass:d1} enforces additional regularity on the term
$F$ in the objective. For example, the bilinear form induced by
the identity $(u,w) \mapsto (u,w)_{L^2(\Omega)}$ does not satisfy
the inequality $(u,w)_{L^2(\Omega)}
\le C\|u\|_{L^1(\Omega)}\|w\|_{L^1(\Omega)}$
for all $u$ and $w$ for any $C > 0$ because $L^1(\Omega)$ is not continuously
embedded into $L^2(\Omega)$. This
means that we restrict to the case
$\tilde{g}^{n-1} = \nabla F(v^{n-1})$
in Line \ref{ln:gtilde} in Algorithm \ref{alg:trm}.
We briefly show that Assumption \ref{ass:d1} nevertheless covers
a large set of problems.
\begin{proposition}\label{prp:d1_satisfied}
Let $z \in L^2(\Omega)$.
If $F(v) \coloneqq \frac{1}{2}\|Kv - z\|_{L^2(\Omega)}^2$
for a bounded linear operator $K : L^1(\Omega) \to L^2(\Omega)$
and all $v \in L^2(\Omega)$, then Assumption \ref{ass:d1} is satisfied.
\end{proposition}
\begin{proof}
Let $K^*$ denote the adjoint operator of $K$.
A straightforward calculation shows that $\nabla F(v) = K^*(Kv - z)$
for all $v \in L^2(\Omega)$ and that
$\nabla \F^2(\xi)(u,w) = (K^*Ku,w)_{L^2(\Omega)} = (Ku,Kw)_{L^2(\Omega)}$.
Next, the Cauchy--Schwarz inequality yields
\[ |\nabla \F^2(\xi)(u,w)| \le \|Ku\|_{L^2(\Omega)}\|Kw\|_{L^2(\Omega)}
\le \|K\|_{1,2}^2\|u\|_{L^1(\Omega)}\|w\|_{L^1(\Omega)}, \]
where $\|K\|_{1,2} = \sup_{\|u\|_{L^1(\Omega)} \le 1}\|Ku\|_{L^2(\Omega)}$,
which is finite by assumption. This shows the claim.
\end{proof}

\begin{remark}
Many integral operators and solution operators
of differential equations have the required regularity. In particular, our
computational example satisfies the prerequisites of Proposition
\ref{prp:d1_satisfied}, as is shown in section \ref{sec:experiments}.
\end{remark}

Our main results on the asymptotics of Algorithm
\ref{alg:trm} are shown under
Assumptions \ref{ass:f_setting} and \ref{ass:d1}.

\subsection{Preliminary Results on Functions of Bounded Variation}\label{sec:bv}

Feasible points for \eqref{eq:p} with finite objective value
are functions of bounded variation that may attain only finitely
many different values. Thus, they are also so-called $SBV$ functions
($S$ for special). In particular, their derivatives are absolutely
continuous measures with respect to $\Ha^{d-1}$ 
\cite{ambrosio2000functions}.

For $d = 1$, that is, if $\Omega = (t_0,t_f)$
for $t_0$, $t_f \in \R$, the $\TV$ term can be analyzed and characterized
with the help of the analysis of the so-called pointwise variation
\cite[Sect,\,3.2]{ambrosio2000functions}. For $t \coloneqq (t_1,\ldots,t_{N-1})^T$ with $t_i \le t_{i+1}$ for all $i \in \{0,\ldots,N-1\}$
and the choice $t_N \coloneqq t_f$ as well as
$a \coloneqq (a_1,\ldots,a_N)^T$ we introduce the notation
\begin{gather}\label{eq:step_representation_def}
v^{t,a} \coloneqq \chi_{(t_0,t_1)} a_1 + \sum_{i=1}^{N-1} \chi_{[t_i,t_{i+1})} a_{i+1}.
\end{gather}
We summarize the relationship between $v$ and $v^{t,a}$
in the following proposition.
\begin{proposition}\label{prp:vta_representation}
Let $v \in \F_{\eqref{eq:p}} \cap BV((t_0,t_f))$. Then there exist
$t \in \R^{N-1}$ and
$a = (a_1,\ldots,a_N)^T \in \{\nu_1,\ldots,\nu_M\}^N$
with $t_0 < \ldots < t_N = t_f$ and $a_i \neq a_{i+1}$
for all $i \in \{1,\ldots,N-1\}$ such that $v = v^{t,a}$.
\end{proposition}
\begin{proof}
The claim follows from the analysis of the so-called pointwise variation
\cite[Sect,\,3.2]{ambrosio2000functions}.
In particular, for $v \in \F_{\eqref{eq:p}} \cap BV((t_0,t_f))$, there exist
$N \in \N$, $t_0 < \ldots < t_N = t_f$
and $a_1$, $\ldots$, $a_N \in \{\nu_1,\ldots,\nu_M\}$ such that
\begin{align}\label{eq:step_representation}
v &= \chi_{(t_0,t_1)} a_1 + \sum_{i=1}^{N-1} \chi_{[t_i,t_{i+1})} a_{i+1},
\text{ and}\\ 
\label{eq:step_representation_tv}
 \TV(v) &= \sum_{i=1}^{N-1} |a_{i+1} - a_i|
 \le (N - 1)(\nu_M - \nu_1) < \infty.
\end{align}
By dropping $t_i$ and $a_i$ from the vectors $t$ and $a$ if $a_i = a_{i+1}$
we obtain that $a_i \neq a_{i+1}$ for all $i \in \{1,\ldots,N-1\}$.
\end{proof}
Because $a_i \in \{\nu_1,\ldots,\nu_M\}$
for all $i \in \{1,\ldots,N\}$, it holds that
$|a_{i+1} - a_i| \in \{0,\ldots,\nu_M - \nu_1\} \subset \N$.
This implies that for
$v$, $\tilde{v} \in \F_{\eqref{eq:p}} \cap BV((t_0,t_f))$, it holds that
\begin{gather}\label{eq:tv_diff_inz_1d}
\TV(\tilde{v}) - \TV(v) \in \Z.
\end{gather}
Note that \eqref{eq:tv_diff_inz_1d} is generally invalid
for the multidimensional case $d > 2$.

\subsection{First-Order Necessary Conditions for $r$-Optimal Points}\label{sec:first_order}

As we will see below, optimality for the trust-region problem \eqref{eq:tr}
alone is not sufficient to characterize limit points of the iterates produced
by Algorithm \ref{alg:trm} because we cannot exclude the case that the
trust-region radius contracts to zero.
This  is due to the fact that the solution of \eqref{eq:tr} is constrained by
$v(x) \in \{\nu_1,\ldots,\nu_M\}$ a.e., which prevents
us from proving that for $\Delta$ sufficiently small the improvement in the
objective of the trust-region subproblem is bounded from below
by a fraction of $\|\nabla F(v)\|_{L^2(\Omega)}$ and $\Delta$.

Next, we introduce an auxiliary optimization problem, which
is used in our analysis of Algorithm \ref{alg:trm} and
\eqref{eq:tr} and \eqref{eq:p}. In particular, the auxiliary
optimization problem
allows us to state a first-order necessary condition for $r$-optimal
points of \eqref{eq:p}, which we will show is satisfied by the limit
points of Algorithm \ref{alg:trm}.
It builds on the step function representation for
$v \in \F_{\eqref{eq:p}} \cap BV((t_0,t_f))$
given by Proposition \ref{prp:vta_representation}.
Let $a_1,\ldots,a_N \in \{\nu_1,\ldots,\nu_M\}$ be given
such that $a_i \neq a_{i+1}$ for all $i \in \{1,\ldots,N-1\}$.
Then, we seek for the solution of the switching location \eqref{eq:sl}
problem
\begin{gather}\label{eq:sl}
\begin{aligned}
\min_{t_1,\ldots,t_{N-1}}\ & F(v^{t,a})\\
\text{ s.t.\ }&  v^{t,a} = \chi_{(t_0,t_1)} a_1 +
                     \sum_{i=1}^{N-1} \chi_{[t_i,t_{i+1})}a_{i+1},\\
              & t_{i-1} \le t_i \text{ for all } i \in \{1,\ldots,N\}.
\end{aligned}\tag{L}
\end{gather}
The problem \eqref{eq:sl} seeks for the optimum switching locations $t_i$
along the interval $(t_0,t_f)$ such that $v^{t,a}$ attains the
values $a_i$ in the given order. At switching location $t_i$, the function
$v^{t,a}$ switches from value $a_i$ to $a_{i+1}$. The constraint formulation
$t_{i-1} \le t_i$ allows the solution $v^{t,a}$ to actually have fewer
switches than given and skip intermediate $a_j$ by choosing $t_{j-1} = t_j$.

\begin{remark}
We use the condition $a_i \neq a_{i+1}$ for all $i \in \{1,\ldots,N-1\}$
to avoid redundancy in the formulation, but this assumption is
without loss of generality (merge subsequent intervals with coinciding
step heights) and does not affect
the representable functions $v$.
\end{remark}

We highlight that \eqref{eq:sl} is a nonlinear program and is,
in general, nonconvex, even if $F$ is convex.
In addition, differentiability of the objective
with respect to the optimization variables of \eqref{eq:sl}
is also not available in the classical sense, which requires
extra care when deriving  necessary optimality conditions
for \eqref{eq:sl} and \eqref{eq:p} below.

Feasible points of \eqref{eq:sl} may be considered as a subset of feasible 
points of \eqref{eq:p}, and \eqref{eq:sl} admits a minimizer, which
is shown in Proposition \ref{prp:sl_necessary}.
\begin{proposition}\label{prp:sl_necessary}
Let Assumptions \ref{ass:f_setting} hold.
Let $a_1$,$\ldots$,$a_N \in \{\nu_1,\ldots,\nu_M\}$
with $a_i \neq a_{i+1}$ for $i \in \{1,\ldots,N-1\}$ be given.
Then,
\begin{enumerate}
\item for any feasible $t \in \F_{\eqref{eq:sl}}$, it follows
that $v^{t,a} \in \F_{\eqref{eq:p}} \cap BV((t_0,t_f))$, and
\item problem \eqref{eq:sl} admits a minimizer.
\end{enumerate}
\end{proposition}
\begin{proof}
The first claim is immediate by virtue of \eqref{eq:step_representation_tv}
and $a_i \in \{\nu_1,\ldots,\nu_M\} \subset \Z$ for all
$i \in \{1,\ldots,M\}$.
We consider a minimizing sequence $(t^n)_n \in \F_{\eqref{eq:sl}}$
of switch locations for \eqref{eq:sl}.
Because $(t^n)_n \subset \F_{\eqref{eq:sl}}$ is bounded, it has a
cluster point;
and we reduce to $t^n \to t$ (after potentially passing to a subsequence).
Clearly, $\TV(v^{s,a}) \le \sum_{i=1}^{N-1} |a_{i+1} - a_i|$ holds for all
$s \in \F_{\eqref{eq:sl}}$, which implies that $(v^{t^n,a})_n$ is
bounded in $BV((t_0,t_f))$. The coordinate-wise convergence $t^n \to t$
implies that $\chi_{[t^n_{i-1},t^n_i)} \to \chi_{[t_{i-1},t_i)}$
in $L^1((t_0,t_f))$ for all $i \in \{1,\ldots,N\}$, which implies
$v^{t^n,a} \to v^{t,a}$ in $L^1((t_0,t_f))$ and in turn in $L^2((t_0,t_f))$.
The lower semi-continuity of $F$ implies
$F(v^{t,a}) \le \liminf_{n\to\infty} F(v^{t_n,a})$,
and consequently $t$ minimizes \eqref{eq:sl}, which proves the second claim.
\end{proof}

\begin{remark}
Because minimizers of \eqref{eq:tr} and \eqref{eq:p} have
finite total variation, the representation of Proposition 
\ref{prp:vta_representation} applies.
\end{remark}

As noted above, the first-order optimality condition for
\eqref{eq:sl} exhibits a subtlety. If we assume that
$\nabla F(v) \in C([t_0,t_f])$ for all $v$,
then the duality between $C([t_0,t_f])$
and the space of Radon measures ensures that differentiation of the
objective with respect to the vector $(t_1,\ldots,t_{N-1}) \in \R^{N-1}$
for $t_1 < \ldots < t_{N-1}$ is well defined. Moreover, using the
chain rule, we obtain the first-order condition
\[ \nabla F(v)(t_i) = 0 
\text{ for } i \in \{1,\ldots,N-1\}.\]

However, $\nabla F(v)$ is not continuous in general
but only an $L^2$-function instead. To provide a first-order condition,
we consider so-called Dini derivatives \cite{giorgi1992dini}.
They allow us to deduce information
on $\nabla F(v)$  in a neighborhood of $t_i$ if the
expression $\nabla F(v)(t_i)$ is not defined.

Let $v \in \F_{\eqref{eq:p}}$, and let $t_i$ be a switching location
of $v$; that is,
$\lim_{t\uparrow t_i}v(t) - \lim_{t\downarrow t_i}v(t)
\in \{\nu_i - \nu_j\,|\, i,j \in \{1,\ldots,M\}\} \setminus\{0\}$.
We define $v_i^- \coloneqq \lim_{t \uparrow t_i} v(t)$,
$v_i^+ \coloneqq \lim_{t \downarrow t_i} v(t)$. 
We consider the Dini derivatives of
the absolutely continuous function
$t \mapsto \int_{t_i}^{t} \nabla F(v)(s)\,\dd s$
at $t_i$:
\begin{align*}
	\overline{D_{i}^+} (\nabla F(v))
	&\coloneqq
	 \limsup_{h \downarrow 0}
	 \frac{1}{h}\int_{t_i}^{t_i+h}
	 \nabla F(v)(s)\,\dd s,\\
	\overline{D_{i}^-} (\nabla F(v))
	&\coloneqq
	 \limsup_{h \downarrow 0}
	 \frac{1}{h}\int_{t_i - h}^{t_i}
	 \nabla F(v)(s)\,\dd s,\\	 
	\underline{D_{i}^+} (\nabla F(v))
	&\coloneqq
	 \liminf_{h \downarrow 0}
	 \frac{1}{h}\int_{t_i}^{t_i+h}
	 \nabla F(v)(s)\,\dd s,\\	 
	\underline{D_{i}^-} (\nabla F(v))	 
	&\coloneqq
	 \liminf_{h \downarrow 0}
	 \frac{1}{h}\int_{t_i - h}^{t_i}
	 \nabla F(v)(s)\,\dd s.
\end{align*} 
We note that if $t_i$ is in the Lebesgue set of $\nabla F(v)$,
all four terms coincide with $\nabla F(v)(t_i)$.
 
\begin{lemma}\label{lem:diniderivative_decrease_step}
Let $v \in \F_{\eqref{eq:p}}$, and let $t_i$ be a switching location
of $v$. Let $v_i^- < v_i^+$ and
$\underline{D^+_i}(\nabla F(v)) > 0$
or
$\overline{D^-_i}(\nabla F(v)) < 0$,
or let $v_i^+ < v_i^-$ and
$\underline{D^-_i}(\nabla F(v)) > 0$
or $\overline{D^+_i}(\nabla F(v)) < 0$.
Then there exist
$\varepsilon > 0$ and $h_0 > 0$ such that for all $h \le h_0$
there exists $d^h \in L^1((t_0,t_f))$ with $\|d^h\|_{L^1((t_0,t_f))}
= h$, $v + d^h \in \F_{\eqref{eq:p}}$
such that $(\nabla F(v), d^h)_{L^2(\Omega)} \le -\frac{\varepsilon}{2}h$.
\end{lemma}
\begin{proof}
\textbf{Case 1: $v_i^- < v_i^+$.}
If $\underline{D_{i}^+} (\nabla F(v)) > \varepsilon$ holds
for some $\varepsilon > 0$, then there exists
$h_0 > 0$ such that for all $h \le h_0$ it follows that
\begin{align*}
 \int_{t_i}^{t_i + h}
     \nabla F(v)(s)(v^-_i - v(s))\,\dd s
   &= \int_{t_i}^{t_i + h} \nabla F(v)(s)(v^-_i - v(s))\,\dd s\\
   &= (v^-_i - v^+_i)\int_{t_i}^{t_i + h} \nabla F(v)(s)\,\dd s\\
   &< -\frac{\varepsilon}{2} h.
\end{align*}
We choose the function
$d^h \coloneqq \chi_{[t_i,t_i+h)} (v_i^- - v_i^+)$.
Similarly, if $\overline{D_{i}^-} (\nabla F(v)) < -\varepsilon$ holds
for some $\varepsilon > 0$, then
there exists $h_0 > 0$ such that for all $h \le h_0$ it follows that
\begin{align*}
 \int_{t_i - h}^{t_i}
     \nabla F(v)(s)(v^+_i - v(s))\,\dd s
   &< -\frac{\varepsilon}{2}h.
\end{align*}
We choose the function $d^h \coloneqq \chi_{[t_i-h,t_i)} (v_i^+ - v_i^-)$.

\textbf{Case 2: $v_i^+ < v_i^-$.}
If $\overline{D_{i}^+} (\nabla F(v)) < -\varepsilon$ holds
for some $\varepsilon > 0$, then
there exists $h_0 > 0$ such that for all $h \le h_0$ it follows that
\begin{align*}
 \int_{t_i}^{t_i + h}
     \nabla F(v)(s)(v^-_i - v(s))\,\dd s
   &< -\frac{\varepsilon}{2} h.
\end{align*}
We choose the function
$d^h \coloneqq \chi_{[t_i,t_i+h)} (v_i^- - v_i^+)$.
Similarly, if
$\underline{D_{i}^-} (\nabla F(v)) > \varepsilon$ holds
for some $\varepsilon > 0$,
then there exists $h_0 > 0$ such that for all $h \le h_0$
it holds that
\begin{align*}
 \int_{t_i - h}^{t_i}
     \nabla F(v)(s)(v^+_i - v(s))\,\dd s
   &< -\frac{\varepsilon}{2} h.
\end{align*}
We choose the function $d^h \coloneqq \chi_{[t_i-h,t_i)} (v_i^+ - v_i^-)$.
\end{proof}

Lemma \ref{lem:diniderivative_decrease_step} yields a necessary
stationarity condition for local minimizers
of \eqref{eq:sl},
which we define as follows.

\begin{definition}\label{dfn:sl_stationary}
Let $v \in \F_{\eqref{eq:p}} \cap BV((t_0,t_f))$
with representation $v = v^{t,a}$ for $t$ and $a$ as given
by Proposition \ref{prp:vta_representation}.
Let the following conditions hold:
\begin{enumerate}
\setlength\itemsep{.5em}
\item $\overline{D^-_i}(\nabla F(v^{t,a}))
       \ge 0 \ge \underline{D_i^+}(\nabla F(v^{t,a}))$
       if $a_i < a_{i+1}$, and
\item $\underline{D^-_i}(\nabla F(v^{t,a}))
       \le 0 \le \overline{D_i^+}(\nabla F(v^{t,a}))$ if $a_{i+1} < a_i$.
\end{enumerate}
Then we say that $v^{t,a}$ is \ref{eq:sl}-stationary.
\end{definition}

Lemma \ref{lem:diniderivative_decrease_step} implies that
\ref{eq:sl}-stationarity is a necessary condition for local minimizers
of \eqref{eq:sl} under Assumption \ref{ass:d1},
which is proven below.

\begin{lemma}\label{lem:sl_1storder_dini}
Let Assumption \ref{ass:d1} hold.
Let $a = (a_1,\ldots,a_N)^T \in \{\nu_1,\ldots,\nu_M\}^N$
with $a_i \neq a_{i+1}$ for $i \in \{1,\ldots,N-1\}$ be given.
Let $t$ be a local minimizer for \eqref{eq:sl}.
Then $v^{t,a}$ is \ref{eq:sl}-stationary.
\end{lemma}
\begin{proof}
We prove the claim with a contrapositive argument. To this end,
let $v^{t,a}$ not be \ref{eq:sl}-stationary.
Then there exists a switching location $t_i$ for which the prerequisites
of Lemma \ref{lem:diniderivative_decrease_step} are satisfied, and we consider
$\varepsilon > 0$ and $d^h$ as asserted by Lemma 
\ref{lem:diniderivative_decrease_step}.
We apply Taylor's theorem (see Proposition \ref{prp:lin_approx}) and obtain
that for some $\xi^h$ in the line segment between $v^{t,a}$ and $v^{t,a} + d^h$
it holds that
\begin{align*}
F(v^{t,a} + d^h) &= F(v^{t,a}) + (\nabla F(v^{t,a}), d^h)_{L^2(\Omega)} + \frac{1}{2}F(\xi^h)(d^h,d^h).
\end{align*}
We use Assumption \ref{ass:d1} to deduce
\[ F(v^{t,a} + d^h) - F(v^{t,a}) \le (\nabla F(v^{t,a}), d^h)_{L^2(\Omega)}
 + C\|d^h\|_{L^1(\Omega)}\|d^h\|_{L^1(\Omega)}
 = (\nabla F(v^{t,a}), d^h)_{L^2(\Omega)} + C h^2.
\]
Inserting the estimate
$(\nabla F(v^{t,a}), d^h)_{L^2(\Omega)} \le -\frac{\varepsilon}{2}h$
from Lemma \ref{lem:diniderivative_decrease_step}, we obtain
\[ F(v^{t,a} + d^h) - F(v^{t,a}) \le -\frac{\varepsilon}{2}h + Ch^2. \]
Because $\varepsilon$ and $C$ do not depend on $h$, we obtain that
\[ F(v^{t,a} + d^h) - F(v^{t,a}) < 0 \]
holds for all $h$ sufficiently small. We restrict $h$ further such that
$t_i + h < t_{i+1}$ and $t_{i-1} < t_i - h$ hold.
This restriction gives that the construction of $d^h$ in the proof
of Lemma \ref{lem:diniderivative_decrease_step} satisfies
$v^{t,a} + d^h \in \F_{\eqref{eq:sl}}$.
Consequently, $t$
is not locally optimal for \eqref{eq:sl}, which concludes the proof.
\end{proof}

We now prove that \ref{eq:sl}-stationarity is  a necessary condition
for minimizers of both the trust-region subproblem \eqref{eq:tr}
and the problem \eqref{eq:p}.

\begin{lemma}\label{lem:sl_nec_opt_tr}
If $v \in \F_{\eqref{eq:p}} \cap BV((t_0,t_f))$ minimizes
$(\text{\emph{\ref{eq:tr}}}(v,\nabla F(v),\Delta))$
for some $\Delta > 0$, then $v$ is \ref{eq:sl}-stationary.
\end{lemma}
\begin{proof}
Let $v = v^{t,a}$ for $t$ and $a$ as given by Proposition 
\ref{prp:vta_representation}. We
prove the claim with a contrapositive argument.
Let $v^{t,a}$ not be \ref{eq:sl}-stationary.
Then by virtue of Lemma \ref{lem:diniderivative_decrease_step}
there exist $\varepsilon > 0$ and $h_0 > 0$ such that for all
$h \le h_0$ there exists
$d^h \in L^1(\Omega)$ such that $v^{t,a} + d^h \in \F_{(\text{\ref{eq:tr}}(v,\nabla F(v),\Delta))}$
and $(\nabla F(v^{t,a}), d^h)_{L^2(\Omega)} \le -\frac{\varepsilon}{2}h$.

Now, we choose $h < \min\left\{h_0,\Delta,\frac{1}{2}\min_i\{t_{i+1} - t_i\}\right\}$.
The inequality $h \le \Delta$ implies that the $v^{t,a} + d^h$ are feasible for
$(\text{\ref{eq:tr}}(v,\nabla F(v),\Delta))$. 
The inequality $h \le \frac{1}{2}\min_i\{t_{i+1} - t_i\}$
implies that for the $d^h$ constructed in the proof of Lemma
\ref{lem:diniderivative_decrease_step} it follows that
$\TV(v^{t,a} + d^{h}) = \TV(v^{t,a})$.

Consequently, the optimal objective value of
$(\text{\ref{eq:tr}}(v,\nabla F(v),\Delta))$
is bounded from above by $-\frac{\varepsilon}{2}h$, and therefore $v = v^{t,a}$
does not minimize  $(\text{\ref{eq:tr}}(v,\nabla F(v),\Delta))$.
This proves the claim.
\end{proof}

\begin{lemma}\label{lem:sl_nec_opt_p}
Let $v \in \F_{\eqref{eq:p}} \cap BV(\Omega)$ be $r$-optimal
for \eqref{eq:p} for some $r > 0$ with representation
$v = v^{t,a}$ for $t$ and $a$ as given by Proposition \ref{prp:vta_representation}.
Then $t$ is a local minimizer for \eqref{eq:sl}.
\end{lemma}
\begin{proof}
Let $0 < h < \frac{1}{\max_{1 \le i<j \le M}|\nu_i - \nu_j|N}\min\{\frac{1}{2}\min\{t_{i+1} - t_i\},r\}$.
Then we obtain for all $t^h_i \in (t_i - h, t_i + h)$
for all $i \in \{1,\ldots,N-1\}$ that $\TV(v^{t_h,a}) = \TV(v)$.
Note that by Proposition \ref{prp:vta_representation} $t_{i+1} - t_i > 0$
for all $i$, and thus by construction the intervals $(t_i - h, t_i + h)$
do not intersect.
Consequently, the $r$-optimality gives that $F(v^{t_h,a}) \ge F(v^{t,a})$.
This implies that  $t$  is a local minimizer for \eqref{eq:sl}.
\end{proof}

\begin{remark}\label{rem:dh_cauchy}
The feasible points $v + d^h \in \F_{\eqref{eq:p}} \cap BV((t_0,t_f))$
constructed with the help of Lemma \ref{lem:diniderivative_decrease_step}
allow us to bound the optimal value of the trust-region subproblem
\eqref{eq:tr} away from zero by a fraction of the trust-region
radius if $v$ is not optimal.
Therefore, they may be considered as the analog of Cauchy points that
are used in the analysis of trust-region methods for nonlinear optimizations.
\end{remark}

We summarize the relationship of the stationarity concepts in
the following theorem.
\begin{theorem}\label{thm:summary_sl_stationarity}
Let Assumptions \ref{ass:f_setting} hold.
Let $v \in \F_{\eqref{eq:p}} \cap BV((t_0,t_f))$
with representation
$v = v^{t,a}$ for $t$ and $a$ as given by Proposition \ref{prp:vta_representation}.
Then the following assertions
hold.
\begin{enumerate}
\item If $t$ is optimal for \eqref{eq:sl}
and Assumption \ref{ass:d1} holds, then $v$ is \ref{eq:sl}-stationary.

\item If $v$ is optimal for $(\text{\emph{\ref{eq:tr}}}(v,\nabla F(v),\Delta))$
for some $\Delta > 0$, then $v$ is \ref{eq:sl}-stationary.

\item If
$v$ is $r$-optimal for some $r > 0$, then $t$ is
a local minimizer for \eqref{eq:sl}.
If Assumption \ref{ass:d1} holds, then $v$ is also
\ref{eq:sl}-stationary.

\item Let $F$ be convex. If $v$ is optimal for
$(\text{\emph{\ref{eq:tr}}}(v,\nabla F(v),\Delta))$
for some $\Delta > 0$, then
$v$ is $r$-optimal.
\end{enumerate}
\end{theorem}
\begin{proof}
The first claim follows from Proposition \ref{prp:sl_necessary}
and Lemma \ref{lem:sl_1storder_dini}.
The second claim follows from Lemma \ref{lem:sl_nec_opt_tr}.
The third claim follows from Lemma \ref{lem:sl_nec_opt_p}.
The fourth claim follows from Proposition \ref{prp:tr_convex_sufficient}.
\end{proof}

\begin{remark}
The construction of the function $d^h$ in the proof of Lemma 
\ref{lem:diniderivative_decrease_step} can be considered as the
construction of points of sufficient decrease (Cauchy points)
for \eqref{eq:sl}; see Remark \ref{rem:dh_cauchy}.
Because of the lack of differentiability of the term $\TV$,
we have not found a way to construct such points for the
problem \eqref{eq:tr}, however. Therefore, we cannot prove
that $v$ solving $(\text{\ref{eq:tr}}(v,\nabla F(v),\Delta))$
is a necessary condition for $r$-optimality.
\end{remark}

\begin{remark}
We highlight that the regularity on the Hessian enforced by
Assumption \ref{ass:d1} 
enters the proof of Lemma \ref{lem:sl_1storder_dini} only to enable
the use of Taylor's theorem. The assumption compensates for
the fact that perturbations of the switching locations by $h$
imply linear changes in the measure of the preimages $(d^h)^{-1}(\nu_i - \nu_j)$,
which  result in changes only in the $L^2$-norm of $O(h^{\frac{1}{2}})$.
However, $O(h^s)$ for $s > \frac{1}{2}$ would be necessary in order to ensure that the
remainder term of the first-order Taylor expansion decays fast enough.
\end{remark}

We briefly show that \ref{eq:sl}-stationarity indeed generalizes the usual
first-order condition for \eqref{eq:sl}.

\begin{proposition}
Let $v \in \F_{\eqref{eq:p}} \cap BV((t_0,t_f))$ be given with representation
$v = v^{t,a}$ for $t$ and $a$ as given by Proposition \ref{prp:vta_representation}.
Let $\nabla F(v)$ have a representative that is continuous at $t_i$,
or let $t_i$ be in the Lebesgue set of $\nabla F(v)$
for all $i \in \{1,\ldots,N - 1\}$. Then, \ref{eq:sl}-stationarity
of $v$ is equivalent to
\[ \nabla F(v)(t_i) = 0 \text{ for all } i \in \{1,\ldots,N-1\}. \]
\end{proposition}
\begin{proof}
First, we note that if $\nabla F(v)$ is continuous at $t_i$, then $t_i$
is a Lebesgue point of $\nabla F(v)$.
Moreover, if $t_i$ is a Lebesgue point of $v$, we obtain that
\[ \overline{D^+_i}\nabla F(v)
   = \overline{D^-_i}\nabla F(v)
   = \underline{D^+_i}\nabla F(v)
   = \underline{D^-_i}\nabla F(v)
   = \nabla F(v)(t_i) \]
by definition of the Lebesgue set; see \cite[Chap.\,3.1]{stein2009real}.
Then for all $i \in \{1,\ldots,N-1\}$, the
inequalities in  Definition \ref{dfn:sl_stationary} are
equivalent to $\nabla F(v)(t_i) = 0$.
\end{proof}
\begin{remark}\label{rem:dge2_construction_fail}
We note that the construction of analogs of Cauchy points and
stationarity cannot be generalized directly to the multidimensional
case $d \ge 2$ and more involved geometric constructions are necessary.
\end{remark}

\subsection{Sufficient Decrease Condition and Asymptotics of the Inner Loop}\label{sec:suff_dec}

Before continuing with the analysis of the outer loop, we analyze the 
asymptotics of the inner loop of Algorithm \ref{alg:trm}.
We show that the inner loop terminates finitely unless the current
iterate satisfies a necessary optimality condition 
(\ref{eq:sl}-stationarity) for $r$-optimality of \eqref{eq:p}.
This is shown in Corollary \ref{cor:step_accept}, for which we
require the preparatory lemma below.
\begin{lemma}\label{lem:step_accept}
Let Assumption \ref{ass:d1} hold.
Let $\sigma \in (0,1)$. Let $v \in \F_{\eqref{eq:p}}$.
Let $\Delta^k \downarrow 0$.
For all $k \in \N$, let $v^k$ be
a minimizer of $(\text{\emph{\ref{eq:tr}}}(v,\nabla F(v),\Delta^{k}))$.
Then at least one of the following statements holds true.
\begin{enumerate}
\item The function $v$ is \ref{eq:sl}-stationary.	  
\item There exists $k_0 \in \N$ such that
	  the objective value of \eqref{eq:tr} evaluated at $v^{k_0}$
	  is zero and $v$ is \ref{eq:sl}-stationary.
\item There exists $k_0 \in \N$ such that
      \[ F(v) + \alpha \TV(v)
         - F(v^{k_0}) - \alpha \TV(v^{k_0})
         \ge \sigma \left(
         (\nabla F(v),v - v^{k_0})_{L^2(\Omega)}
         + \alpha \TV(v) - \alpha \TV(v^{k_0}) \right);
      \]
	  that is, the sufficient decrease condition
	  \eqref{eq:1d_accept}	  
	  (Line \ref{ln:suffdec} in Algorithm \ref{alg:trm}) holds.
\end{enumerate}
\end{lemma}
\begin{proof}
We first observe that if the objective value of
$(\text{\ref{eq:tr}}(v,\nabla F(v),\Delta^{k}))$ evaluated at
$v^k$ is zero, then Theorem \ref{thm:summary_sl_stationarity} implies
that $v$ is \ref{eq:sl}-stationary.
Next, we assume that the first two claims do not hold true, and we set forth
to prove the third.
From $\Delta^{k} \downarrow 0$ and the constraints
of \eqref{eq:tr} it follows that 
$v^{k} \to v$ in $L^1(\Omega)$ and in $L^2(\Omega)$.
Moreover, we use the lower semi-continuity
$\TV(v) \le \liminf_{k\to\infty} \TV(v^k)$
to distinguish the following two cases.

\textbf{Case 1: $0 = \lim_{\ell\to\infty} \TV(v^{k_\ell}) - \TV(v)$
for a subsequence $(v^{k_\ell})_\ell$.}
Because \eqref{eq:tv_diff_inz_1d} implies $\TV(v^{k_\ell}) - \TV(v) \in \Z$,
it follows that $\TV(v^{k_\ell}) = \TV(v)$ for all $\ell \ge \ell_1$
for some $\ell_1 \in \N$.
Because the optimal objective of $(\text{TR}(v,\nabla F(v),\Delta^{k_\ell}))$
is negative for all $\ell$, it follows that
\[ (\nabla F(v), v - v^{k_\ell})_{L^2(\Omega)} > 0 \]
for all $\ell \ge \ell_1$. From
Taylor's theorem (see Proposition \ref{prp:lin_approx}) it follows that
\begin{align*}
F(v) - F(v^k) &= \sigma(\nabla F(v), v - v^{k_\ell})_{L^2(\Omega)} + 
                 (1 - \sigma)(\nabla F(v), v - v^{k_\ell})_{L^2(\Omega)}
                 + \frac{1}{2}\nabla^2F(\xi^{k_{\ell}})(v - v^{k_\ell},v - v^{k_\ell})
\end{align*}
for some $\xi^{k_\ell}$ in the line segment between $v$ and $v^{k_\ell}$.
Because $v$ is not \ref{eq:sl}-stationary by assumption, we can apply Lemma 
\ref{lem:diniderivative_decrease_step} and the optimality of
$v^{k_\ell}$ for the problem
$(\text{\ref{eq:tr}}(v,\nabla F(v),\Delta^{k_\ell}))$
as in the proof of Lemma \ref{lem:sl_nec_opt_p} to obtain the
bounds
\[ (1 - \sigma)(\nabla F(v), v - v^{k_\ell})_{L^2(\Omega)}
   \ge (1 - \sigma)\frac{\varepsilon}{2}\Delta^{k_\ell}, \]
and
\[ \frac{1}{2}\nabla^2F(\xi^{k_{\ell}})(v - v^{k_\ell},v - v^{k_\ell})
   \le \frac{1}{2}C(\Delta^{k_\ell})^2 \]
for constants $\varepsilon > 0$ and $C > 0$ and all $\ell$ sufficiently large.
This implies that the term $(1 - \sigma)(\nabla F(v), v - v^{k_\ell})_{L^2(\Omega)}$
eventually dominates the term $\frac{1}{2}\nabla^2F(\xi^{k_{\ell}})(v - v^{k_\ell},v - v^{k_\ell})$ and that, together with $\TV(v) = \TV(v^{k_\ell})$,
Outcome 3 holds with the choice $k_0 \coloneqq k_\ell$ for some
$\ell$ sufficiently large.

\textbf{Case 2: $\TV(v) - \TV(v^{k_\ell}) < 0$ for a subsequence
$(v^{k_\ell})_\ell$ and all $\ell \ge \ell_1$ for some $\ell_1 \in \N$.}
In this case we note that \eqref{eq:tv_diff_inz_1d} implies
$\TV(v) - \TV(v^{k_\ell}) \le -1$, giving
$\limsup_{\ell\to\infty} \TV(v) - \TV(v^{k_\ell}) \le -1$. This gives
\[ \limsup_{\ell \to \infty}
   \left\{(\nabla F(v), v - v^{k_\ell})_{L^2(\Omega)} + \alpha \TV(v) - \alpha \TV(v^k)\right\}
   = \alpha \limsup_{k\to\infty}\ \TV(v) - \TV(v^k) = -\alpha < 0, \]
where the first identity follows from the Cauchy--Schwarz inequality
and $v^k \to v$ in $L^2(\Omega)$.
However, this means that there exists $k_0 \in \N$ such that
$v$ is $\Delta^{k_0}$-optimal, which means that the first claim holds
and contradicts the assumption that it does not. This concludes the proof.
\end{proof}

\begin{corollary}\label{cor:step_accept}
Let Assumptions \ref{ass:f_setting} and \ref{ass:d1} hold.
Then for all iterations $n = 0,...$ executed
by Algorithm \ref{alg:trm} it follows that one of the following outcomes
holds.
\begin{enumerate}
\item The inner loop terminates after finitely many
iterations, with the outcome that the predicted reduction
is zero (and the current iterate is \ref{eq:sl}-stationary)
or the sufficient decrease condition
\eqref{eq:1d_accept} is satisfied.
\item The inner loop does not terminate, and the current iterate
is \ref{eq:sl}-stationary.
\end{enumerate}
\end{corollary}
\begin{proof}
This follows from Lemma \ref{lem:step_accept}
with $\Delta^k = \Delta^{n,k}$ and $v = v^{n-1}$
\end{proof}

\begin{remark}\label{rem:step_accept}
The proof of Lemma \ref{lem:step_accept}
hinges on the fact that for $\TV(v) - \TV(\tilde{v}) < 0$
for feasible $v$, $\tilde{v}$ we can deduce that
$\TV(v) - \TV(\tilde{v}) \le -1$ from \eqref{eq:tv_diff_inz_1d}.
This is due to two facts.
First, a discrete-valued function of bounded variation is a
step function defined on finitely many intervals.
Second, considering the step function representation
\eqref{eq:step_representation},
we obtain that $\Ha^{d-1}(E) = \Ha^0(E) = 1$ for an interior facet
(that is, a point $E = \{x_i\}$ for $i \in \{1,\ldots,N-1\}$)
that connects the intervals $[x_{i-1},x_i)$ and $[x_i,x_{i+1})$.
\end{remark}

\begin{remark}
Because $\Ha^{d-1}(E)$ is not the counting measure
for $d \ge 2$ but can attain arbitrarily small values instead,
\eqref{eq:tv_diff_inz_1d} is in general incorrect for $d \ge 2$,
and $\TV(v) - \TV(\tilde{v}) < 0$ does not imply
$\TV(v) - \TV(\tilde{v}) < -\varepsilon$ for some
$\varepsilon > 0$.
Hence, the proof cannot be transferred directly to domains
in higher dimensions.
\end{remark}

\subsection{Asymptotics of Algorithm \ref{alg:trm}}\label{sec:asymptotics}

With the help of the concept of \ref{eq:sl}-stationarity,
we are able to characterize the asymptotics of the sequence
of iterates generated by Algorithm \ref{alg:trm}.
We recall that we say that a sequence
$(v^n)_n \subset BV(\Omega)$ converges strictly to a limit $v$
in $BV(\Omega)$ if $v^n \to v$ in $L^1(\Omega)$ and
$\TV(v^n) \to \TV(v)$.

\begin{theorem}\label{thm:pure_tr_asymptotics}
Let Assumptions \ref{ass:f_setting} and \ref{ass:d1} hold.
Let the iterates $(v^n)_n$ be produced by Algorithm \ref{alg:trm}.
Then all iterates are feasible for \eqref{eq:p}, and 
the sequence of objective values $(J(v^n))_n$
is monotonously decreasing.
Moreover, one of the following mutually exclusive outcomes holds:
\begin{enumerate}
\item\label{itm:finite_seq_tr} The sequence $(v^n)_n$ is finite.
The final element $v^N$ of $(v^n)_n$ solves
$(\text{\emph{\ref{eq:tr}}}(v^N,\nabla F(v^N),\Delta))$
for some $\Delta > 0$, and $v^N$ is \ref{eq:sl}-stationary.
\item\label{itm:finite_seq_tr_contract} The sequence $(v^n)_n$ is finite, and
the inner loop does not terminate for the final element $v^N$,
which is \ref{eq:sl}-stationary.
\item\label{itm:infinite_seq_sl} The sequence $(v^n)_n$ has a weak-$^*$ accumulation
point in $BV((t_0,t_f))$. Every weak-$^*$ accumulation point of $(v^n)_n$
is feasible, \ref{eq:sl}-stationary and strict.
If the trust-region radii are bounded away from zero, that is,
if $\liminf_{n_\ell\to\infty} \min_{k}\Delta^{n_{\ell}+1,k} > 0$
for a subsequence $(v^{n_\ell})_\ell$ and $\bar{v}$ is a weak-$^*$ accumulation point of $(v^{n_\ell})_\ell$, then 
$\bar{v}$ solves $(\text{\emph{\ref{eq:tr}}}(\bar{v},\nabla F(\bar{v}),\bar{\Delta}))$ for some $\bar{\Delta} > 0$.
\end{enumerate}
\end{theorem}
\begin{proof}
First, we note that a new iterate $v^n$ is produced (accepted)
by the inner loop of Algorithm \ref{alg:trm} if the condition 
$\pred(v^{n-1},\Delta^{n,k}) > 0$ and 
the sufficient decrease condition \eqref{eq:1d_accept}
are satisfied. Combining these estimates, we have
$J(v^n) = F(v^n) + \alpha\TV(v^n) < \TV(v^{n-1}) + \alpha\TV(v^{n-1})
= J(v^{n-1})$. This implies that the sequence $(J(v^n))_n$ is monotonously 
decreasing. Because the feasible sets of all trust-region subproblems
\eqref{eq:tr} are included in the feasible set of the problem \eqref{eq:p},
it follows inductively from Proposition \ref{prp:tr_well_defined}
that Line \ref{ln:trstep} produces $\bar{v}^{n,k} \in \F_{\eqref{eq:p}}$.

We first consider Outcome \ref{itm:finite_seq_tr}. We observe that
Algorithm \ref{alg:trm} terminates with iterate $v^{n-1}$ if and only if
for some outer iteration $n \in \N$ and inner iteration $k \in \N$
it holds that $\pred(v^{n-1},\Delta^{n,k}) = 0$.
The definition of $\pred(v^{n-1},\Delta^{n,k})$ in Lines \ref{ln:trstep}
and \ref{ln:pred} implies that $v^N \coloneqq v^{n-1}$ solves
$(\text{\ref{eq:tr}}(v^N,\nabla F(v^N),\Delta))$ for some $\Delta > 0$.

Next we consider Outcome \ref{itm:finite_seq_tr_contract}.
If the inner loop does not terminate finitely, Lemma \ref{lem:step_accept} 
yields that the current iterate $v^N = v^{n-1}$
is \ref{eq:sl}-stationary.

Therefore, it suffices to assume that Outcome \ref{itm:finite_seq_tr}
and Outcome \ref{itm:finite_seq_tr_contract} do not hold true
and to prove that Outcome \ref{itm:infinite_seq_sl} holds in
this case. The absence of Outcome \ref{itm:finite_seq_tr} and
Outcome \ref{itm:finite_seq_tr_contract} implies that for all
$n \in \N$ the inner loop terminates after finitely
many iterations and Algorithm \ref{alg:trm} produces an
infinite sequence of iterates $(v^n)_n$.

The sequence $(\TV(v^n))_n$ is bounded from below by zero
and from above because the sequence $(\frac{1}{\alpha} J(v^n))_n$ is decreasing and the sequence
$(F(v^n))_n$ is bounded from below. Therefore, the sequence
$(v^n)_n$ has a weak-$^*$ accumulation point in $BV((t_0,t_f))$. Lemma \ref{lem:discreteness_feasibility}
yields that all weak-$^*$ accumulation points are feasible for \eqref{eq:p}.

Next we show that weak-$^*$ accumulation points are strict.
We execute a contrapositive argument and assume that
$v$ is a weak-$^*$ accumulation point of $(v^n)_n$
with approximating subsequence $v^{n_\ell} \weakstarto v$
in $BV((t_0,t_f))$ that is not strict. This implies
that there exists a subsequence (for ease of notation
also denoted by $v^{n_\ell}$) such that
$\TV(v) < \liminf_{\ell \to \infty} \TV(v^{n_\ell})$.
We define
$\delta \coloneqq \liminf_{\ell \to \infty}\TV(v^{n_\ell}) - \TV(v)$.
We deduce that $\delta \ge 1$ from \eqref{eq:tv_diff_inz_1d}.
The sufficient decrease condition \eqref{eq:1d_accept}
on acceptance combined with the optimality of the solution
of Line \ref{ln:trstep} gives
\begin{gather}\label{eq:suffdec_atv_bnd}
\begin{aligned}
\frac{1}{\sigma}\ared(v^{n_\ell}, v^{n_\ell + 1,k}) 
&\ge (\nabla F(v^{n_\ell}),v^{n_\ell} - v^{n_\ell+1,k})_{L^2(\Omega)}
   + \alpha \TV(v^{n_\ell}) - \alpha \TV(v^{n_\ell+1,k})\\
&\ge (\nabla F(v^{n_\ell}), v^{n_\ell} - v)_{L^2(\Omega)}
   + \alpha \delta
\end{aligned}
\end{gather}
for infinitely many elements $\ell$ of the subsequence
and all inner iterations $k$, for which
$v$ is feasible for the trust-region subproblem.
For all inner iterations $k$ we obtain the
estimate
\[|(\nabla F(v^{n_\ell}),v^{n_\ell} - \tilde{v})_{L^2(\Omega)}|
\le \sqrt{\max_{i,j} |\nu_i - \nu_j|} \sqrt{
\Delta^{n_{\ell+1,k}}}
\|\nabla F(v^{n_\ell})\|_{L^2(\Omega)}\]
for all $\tilde{v}$ feasible for the trust-region
subproblem by virtue of the Cauchy--Schwarz inequality.
Thus there exists $k_0$ such that
for all inner iterations $k \ge k_0$ we have
$|(\nabla F(v^{n_\ell}),v^{n_\ell} - \tilde{v})_{L^2(\Omega)}|
\le \tfrac{1 - \sigma}{3 - \sigma}\alpha$
and $|F(v^{n_\ell}) - F(\tilde{v})|
\le \tfrac{1 - \sigma}{3 - \sigma}\alpha$.

Because $v^{n_\ell} \weakstarto v$ in 
$BV((t_0,t_f))$, we have $v^{n_\ell} \to v$ in $L^2((t_0,t_f))$ and
$(\nabla F(v^{n_\ell}), v^{n_\ell} - v)_{L^2(\Omega)} \to 0$
after potentially restricting to another subsequence,
implying that $v$ is feasible in iteration $k_0$ for all
$\ell \ge \ell_0$ for some $\ell_0$ large enough.
We deduce $\ared(v^{n_\ell}, v^{n_\ell + 1,k_0})
\ge \pred(v^{n_\ell},\Delta^{n_{\ell+1},k_0}) - 2 \alpha \frac{1 - \sigma}{3 - \sigma}$
and $\pred(v^{n_\ell},\Delta^{n_{\ell+1},k_0})
\ge \alpha - \alpha \frac{1 - \sigma}{3 - \sigma}$, implying that 
$\ared(v^{n_\ell}, v^{n_\ell + 1,k_0})
\ge \sigma\pred(v^{n_\ell},\Delta^{n_{\ell+1},k_0})$.
Thus a new step is accepted latemost
in inner iteration $k_0$ for all $\ell \ge \ell_0$.
Combining these results, we obtain that
$J(v^{n_\ell + 1}) \to -\infty$
for $\ell \to \infty$, which contradicts that $J$ is bounded
from below. Thus, every weak-$^*$ accumulation point of $(v^{n})_n$
is strict.

Next we prove that the weak-$^*$ and strict limit $v$ solves
$(\text{\ref{eq:tr}}(v,\nabla F(v),\Delta))$
for some $\Delta > 0$ if the trust-region radii upon acceptance
of a subsequence of $(v^{n_\ell})_\ell$ are bounded away from zero.
We restrict ourselves to such a subsequence and denote it also 
by $(v^{n_\ell})_\ell$ for ease of notation.
A lower bound on the trust-region radii upon acceptance is
$\underline{\Delta} \coloneqq \inf_{\ell\in\N} \min_{k} \Delta^{n_\ell+1,k} > 0$.
Because we may assume that $v^{n_\ell} \to v$ in $L^2((t_0,t_f))$
by restricting to another subsequence (also denoted by $v^{n_\ell}$),
it follows that $\nabla F(v^{n_\ell}) \to \nabla F(v)$ in $L^2((t_0,t_f))$.
Moreover, we restrict this subsequence further (also denoted by $v^{n_\ell}$)
such that $\lim_{\ell \to \infty} \TV(v^{n_\ell}) = \TV(v)$.
Combining these observations with \eqref{eq:tv_diff_inz_1d}, 
we deduce that there exists $\ell_0 \in \N$
such that for all $\ell \ge \ell_0$ we have
$\TV(v^{n_\ell}) = \TV(v)$.
Using this identity, we obtain
$\pred(v, \underline{\Delta})
 \le\pred(v^{n_\ell}, \min_k \Delta^{n_\ell + 1, k}) \to 0$.
The sequence of predicted reductions on acceptance
of the step $(\pred(v^{n_\ell}, \min_k \Delta^{n_\ell + 1, k}))_\ell$
tends to zero because otherwise we would obtain the contradiction
$J(v^{n_\ell+1}) \to -\infty$.
This implies that $\pred(v, \underline{\Delta}) = 0$, and consequently
$v$ solves $(\text{\ref{eq:tr}}(v,\nabla F(v),\underline{\Delta}))$
in this case.

Next we prove that $v$ is \ref{eq:sl}-stationary.
Again we consider a subsequence $(v^{n_\ell})_\ell$ such that
$v^{n_\ell}\weakstarto v$ that is also strictly convergent.
We restrict $(v^{n_\ell})_\ell$ further
to a pointwise a.e.\ convergent 
subsequence, also denoted by $(v^{n_\ell})_\ell$,
such that
$\lim_{\ell \to \infty} \TV(v^{n_\ell}) = \TV(v)$.
This implies that there exists $\ell_0 \in \N$ such that
for all $\ell \ge \ell_0$ we have $\TV(v^{n_\ell}) = \TV(v)$.

Because $\TV(v) < \infty$, we can use the representation
given by Proposition \ref{prp:vta_representation} and
deduce that there are $N \in \N$, $a \in \{\nu_1,\ldots,\nu_M\}^N$
with $a_i \neq a_{i+1}$ for all $i \in \{1,\ldots,N-1\}$,
and $0 = t_0 < \ldots < t_N = t_f$ such that
\[ v = v^{t,a} = \chi_{(t_0,t_1)} a_1
+ \sum_{i=1}^{N-1} \chi_{[t_i,t_{i+1})} a_{i+1}. \]

Having obtained this characterization, we use a contradiction
argument to prove that $v$ is \ref{eq:sl}-stationary by assuming that
the claim does not hold true, invoking Lemma 
\ref{lem:diniderivative_decrease_step} and obtaining a contradiction.
To this end, we assume that \ref{eq:sl}-stationarity is violated
for some switching location $t_i$ of $v$.
This implies that the prerequisites of Lemma \ref{lem:diniderivative_decrease_step}
are satisfied, and we obtain that there exist 
$\varepsilon > 0$ and $h_0 > 0$ such that for all $h \le h_0$,
there exists $d^h \in L^1((t_0,t_f))$ with $\|d^h\|_{L^1(\Omega)} = h$,
$v + d^h \in \F_{\eqref{eq:p}}$, and
$(\nabla F(v), d^h)_{L^2(\Omega)} \le -\frac{\varepsilon}{2}h$.
Moreover, by construction of $d^h$ in the
proof of Lemma \ref{lem:diniderivative_decrease_step} 
and choosing $h_0 < \frac{1}{2}\min_i\{t_{i+1} - t_i\}$
it also follows that $\TV(v + d^h) = \TV(v) = \TV(v^{n_\ell})$.

The remaining argument is the following. We first show \textbf{Part 1}
that we obtain a contradiction if the trust-region radius upon acceptance of
the iterates $n_\ell + 1$ is larger than $2 h$ for some $h \le h_0$
infinitely often. Then, we show \textbf{Part 2} that the trust-region radius
cannot contract to zero for the subsequence $(v^{n_\ell})_\ell$.
Note that $\varepsilon$ and $h_0$ chosen above depend on the limit
of the subsequence. 

\textbf{Part 1:} To establish the contradiction, we assume that
trust-region radius upon acceptance of the iterates $n_\ell + 1$ is
larger than $2 h$ for some $h \le h_0$ infinitely often. To simplify
the argument, we restrict it
to the subsequence of $(n_\ell)_\ell$, for ease of notation denoted by the 
same symbol, such that the trust-region radius upon acceptance of the
iterates $n_\ell + 1$ is larger than $2 h$. Because $v^{n_\ell} \to v$
in $L^1(\Omega)$, there exists $\ell_1 \ge \ell_0$ such that for all $\ell \ge \ell_1$ it holds that $\|v^{n_\ell} - v\|_{L^1(\Omega)} \le h$.
Thus, for iteration $n_\ell + 1$ and all inner iterations $k$, we obtain
\[ \min\ (\text{TR}(v^{n_\ell},\nabla F(v^{n_\ell}),\Delta^{n_\ell+1,k}))
   \le 
   \left\{
   \begin{aligned}  
   (\nabla F(v^{n_\ell}), v + d^{\Delta^{n_\ell + 1,k}/2} - v^{n_\ell})_{L^2(\Omega)}
   &\text{ if }\Delta^{n_\ell + 1,k} \le 2 h_0, \\
   (\nabla F(v^{n_\ell}), v + d^{h_0} - v^{n_\ell})_{L^2(\Omega)}
   &\text{ if } \Delta^{n_\ell + 1,k} \ge 2 h_0.
   \end{aligned}
   \right.
\]

For the accepted iterate, we obtain the estimate
\begin{align*}
   J(v^{n_\ell}) - J(v^{n_\ell + 1})
   &\ge \sigma
   (\nabla F(v^{n_\ell}), v^{n_\ell} - v - d^{h})_{L^2(\Omega)}\\
   &= -\sigma(\nabla F(v), d^{h})_{L^2(\Omega)}
     +\sigma(\nabla F(v) - \nabla F(v^{n_\ell}), d^{h})_{L^2(\Omega)}
     +\sigma(\nabla F(v^{n_\ell}), v^{n_\ell} - v)_{L^2(\Omega)}.
\end{align*}
Because $v^{n_\ell} \to v$, the last two terms tend to zero,
and there exists $\ell_2 \ge \ell_1$ such that for all $\ell \ge \ell_2$
it holds that
\[ J(v^{n_\ell}) - J(v^{n_\ell + 1})
   \ge -\frac{\sigma}{2}(\nabla F(v), d^{h})_{L^2(\Omega)} \ge \frac{1}{4}\sigma\varepsilon h. \]
This implies the contradiction $J(v^{n_\ell + 1}) \to -\infty$.

\textbf{Part 2:} We show that the trust-region radii upon acceptance
of the iterates $v^{n_\ell + 1}$ cannot contract to zero.
Let $0 < h \le h_0$ be given. Then there exists $\ell_3 \ge \ell_0$
such that for all $\ell \ge \ell_3$ it holds that
$\|v - v^{n_\ell}\|_{L^2(\Omega)} \le h^2$ and
$\|\nabla F(v) - \nabla F(v^{n_\ell})\|_{L^2(\Omega)} \le h^2$.
Moreover, there exists a constant
$c_1 > 0$ that can be chosen independently of $h$ and $\ell$
such that
\begin{gather}\label{eq:vdh_vnell_lip_estim}
\left|(\nabla F(v),d^h)_{L^2(\Omega)}
- (\nabla F(v^{n_\ell}),v + d^h - v^{n_\ell})_{L^2(\Omega)}\right|
\le h^2 \|d^h\|_{L^2(\Omega)} + \|\nabla F(v^{n_\ell})\|_{L^2(\Omega)} h^2
\le c_1 h^2.
\end{gather}

Let $\Delta^* \in \{ \Delta^0 2^{-j}\,|\,j \in \N\}$.
Let $v^{*}$ denote a minimizer of the problem
$(\text{\ref{eq:tr}}(v^{n_\ell},\nabla F(v^{n_\ell}),\Delta^*))$
with the fixed extra constraint $\TV(\tilde{v}) = \TV(v)$.
Then Taylor's theorem (see Proposition \ref{prp:lin_approx}) 
and Assumption \ref{ass:d1} imply
\begin{align*}
F(v^{n_\ell}) - F(v^{*})
\ge -(\nabla F(v^{n_\ell}),v^{*} - v^{n_\ell})_{L^2(\Omega)}
-c_2 \|v^{*} - v^{n_\ell}\|_{L^1(\Omega)}^2
\end{align*}
for some $c_2 > 0$. 
We can choose $\Delta^*$ sufficiently small (choose $j$ sufficiently large)
such that the inequalities
$\Delta^* \le h_0$, $(1-\sigma)\frac{\varepsilon}{4}\Delta^*
-(\Delta^*)^2(0.25 c_1 + c_2) > 0$, and $(\|\nabla F(v)\|_{L^2(\Omega)} + (\Delta^*)^2)(\max_i \nu_i - \min_i \nu_i)\sqrt{\Delta^*} \le \alpha$ hold true.

To combine these considerations, we choose the $\ell_3 \in \N$
introduced above such that \eqref{eq:vdh_vnell_lip_estim} holds with
the choice $h = \Delta^*$. For all $\ell \ge \ell_3$, we deduce
\begin{align*}
-(1-\sigma)(\nabla F(v^{n_\ell}),v^{*} - v^{n_\ell})_{L^2(\Omega)}
-c_2 \|v^{*} - v^{n_\ell}\|^2_{L^1(\Omega)}
&\ge -(1-\sigma)(\nabla F(v^{n_\ell}),v^{n_\ell} + d^{\Delta^*/2} - v^{n_\ell})_{L^2(\Omega)}
-c_2 (\Delta^*)^2\\
&\ge -(1-\sigma)(\nabla F(v),d^{\Delta^*/2})_{L^2(\Omega)}
-(\Delta^*)^2(0.25 c_1 + c_2) \\
&\ge (1-\sigma)\frac{\varepsilon}{4}\Delta^*
-(\Delta^*)^2(0.25 c_1 + c_2) \ge 0,
\end{align*}
where we have used \eqref{eq:vdh_vnell_lip_estim} with $h = \Delta^*$
in the second inequality.
It follows that
\begin{gather}\label{eq:pred_deltastar}
-(\nabla F(v^{n_\ell}),v^{*} - v^{n_\ell})_{L^2(\Omega)}
\ge 
-(\nabla F(v^{n_\ell}),v + d^{\Delta^* / 2} - v^{n_\ell})_{L^2(\Omega)}
\ge \frac{\varepsilon}{4} \Delta^* - c_1 (\Delta^*)^2
> 0.
\end{gather}
Moreover, for all $\tilde{v}$ feasible for $(\text{\ref{eq:tr}}(v^{n_\ell},\nabla F(v^{n_\ell}),\Delta^*))$
we have 
\[ (\nabla F(v^{n_\ell}), \tilde{v} - v^{n_\ell})_{L^2(\Omega)} \le \|\nabla F(v^{n_\ell})\|_{L^2(\Omega)}\|\tilde{v} - v^{n_\ell}\|_{L^2(\Omega)}
\le (\|\nabla F(v)\|_{L^2(\Omega)} + (\Delta^*)^2)(\max_i \nu_i - \min_i \nu_i)\sqrt{\Delta^*}
\]
Thus the assumed lower bound on $\alpha$ implies that any solution $\hat{v}^*$
of $(\text{\ref{eq:tr}}(v^{n_\ell},\nabla F(v^{n_\ell}),\Delta^*))$
also satisfies $\TV(\hat{v}^*) \le \TV(v^*) = \TV(v^{n_\ell}) = \TV(v)$.
To complete the argument, we distinguish the cases
\textbf{Case 2a} $\TV(\hat{v}^*) < \TV(v)$ and \textbf{Case 2b}
$\TV(\hat{v}^*) = \TV(v)$.

\textbf{Case 2a:} If $\TV(\hat{v}^*) < \TV(v^{n_\ell})$, then
$\TV(v^{n_\ell}) - \TV(\hat{v}^*) \ge 1$ by \eqref{eq:tv_diff_inz_1d}.
By virtue of Assumption \ref{ass:d1} and
Proposition \ref{prp:discretev_L1_L2_inequality}, we
may choose $\Delta^*$ small enough and
$\ell_4 \ge \ell_3$ such that for all $\ell \ge \ell_4$
it follows that
$|F(\hat{v}^{*}) - F(v^{n_\ell})| < \frac{1 - \sigma}{2}$
and
$|\sigma(\nabla F(v^{n_\ell}), \hat{v}^* - v^{n_\ell})_{L^2(\Omega)}|
< \frac{1 - \sigma}{2}$. We deduce
\begin{align*}
\ared(v^{n_\ell},\hat{v}^{*}) &=
F(\hat{v}^{*}) - F(v^{n_\ell}) + \TV(v^{n_\ell}) - \TV(\hat{v}^*) \\
&\ge -\frac{1 - \sigma}{2} + \TV(v^{n_\ell}) - \TV(\hat{v}^*)
&&\enskip\text{\tiny{$|F(\hat{v}^{*}) - F(v^{n_\ell})| < \frac{1 - \sigma}{2}$}} \\
&\ge -(1 - \sigma) - \sigma(\nabla F(v^{n_\ell}), \hat{v}^* - v^{n_\ell})_{L^2(\Omega)} + \TV(v^{n_\ell}) - \TV(\hat{v}^*) 
&&\enskip\text{\tiny{$|\sigma(\nabla F(v^{n_\ell}), \hat{v}^* - \nabla F(v^{n_\ell}))_{L^2(\Omega)}|
< \frac{1 - \sigma}{2}$}}\\
&\ge - \sigma(\nabla F(v^{n_\ell}), \hat{v}^* - v^{n_\ell})_{L^2(\Omega)} + \sigma(\TV(v^{n_\ell}) - \TV(\hat{v}^*)) 
&&\enskip\text{\tiny{$\TV(v^{n_\ell}) - \TV(\hat{v}^*) \ge 1$}}\\
&= \sigma \pred(v^{n_\ell},\hat{v}^{*}).
\end{align*}
Let $j$ be such that $\Delta^* = \Delta^0 2^{-j}$.
Then, this argument implies that at the latest the $j$th iteration of
the inner loop is accepted for all $\ell \ge \ell_4$
if $\TV(\hat{v}^*) < \TV(v)$.

\textbf{Case 2b:} If $\TV(\hat{v}^*) = \TV(v)$, then
we may consider $v^* = \hat{v}^*$, and for all $\ell \ge \ell_3$
it follows that in the inner iteration
$k$ such that $\Delta^0 2^{-k} = \Delta^*$ we have that
\begin{align*}
\ared(v^{n_\ell},v^{*}) &\ge F(v^{n_\ell}) - F(v^{*})\\
&\ge
-\sigma (\nabla F(v^{n_\ell}),v^{*} - v^{n_\ell})_{L^2(\Omega)}
-(1-\sigma)(\nabla F(v^{n_\ell}),v^{*} - v^{n_\ell})_{L^2(\Omega)}
-c_2 \|v^{*} - v^{n_\ell}\|_{L^1(\Omega)}^2\\
&\ge -\sigma (\nabla F(v^{n_\ell}),v^{*} - v^{n_\ell})_{L^2(\Omega)}
 = \sigma\pred(v^{n_\ell}, \Delta^*).
\end{align*} 
Inserting \eqref{eq:pred_deltastar}, it follows that
$\pred(v^{n_\ell},\Delta^{*})
\ge \frac{\varepsilon}{4} \Delta^* - c_1 (\Delta^*)^2 > 0$.
Let $j$ be such that $\Delta^* = \Delta^0 2^{-j}$. The argument
above implies that at the latest the $j$th iteration of the inner loop
is accepted for all $\ell \ge \ell_3$
if $\TV(\hat{v}^*) = \TV(v)$.

Combining these cases means that at the latest the $j$th iteration
of the inner loop is accepted for all $\ell \ge \ell_4$, which
implies that the trust-region radius upon acceptance
is eventually always greater than or equal to
$\Delta^*$, which settles \textbf{Part 2} and
implies $J(v^{n_{\ell+1}})\to-\infty$
as established in \textbf{Part 1}.
This completes the contradiction under the assumption that
$v$ is not \ref{eq:sl}-stationary. Consequently, $v$ has to be
\ref{eq:sl}-stationary, which completes the proof.
\end{proof}

\begin{corollary}
Under the prerequisites of Theorem \ref{thm:pure_tr_asymptotics}
it holds that the final iterate produced by Algorithm \ref{alg:trm}
and---in the absence of finite termination---all
limit points of the iterates produced by Algorithm \ref{alg:trm}
are \ref{eq:sl}-stationary.
\end{corollary}
\begin{proof}
This follows from Theorem \ref{thm:pure_tr_asymptotics} and
Theorem \ref{thm:summary_sl_stationarity}.
\end{proof}

\begin{remark}
In particular, we obtain that if the trust-region radius
is bounded away from zero or the algorithm terminates finitely,
then the limit is not only \ref{eq:sl}-stationary
but also a minimizer of the trust-region subproblem for a 
strictly positive trust-region radius.
\end{remark}

\begin{remark}\label{rem:active_set_idea}
This analysis hinges on the fact that after finitely many iterations
(and potentially passing to a subsequence),
the value of $\TV(v^n)$ does not change anymore for the iterates $v^n$
produced by Algorithm \ref{alg:trm}.
Once this has happened, Assumption \ref{ass:d1} and Taylor's theorem are
employed to prove that the the limit points of $(v^n)_n$ are
\ref{eq:sl}-stationary.

Compared with algorithms for nonlinear optimization, this can be viewed as
the active set settling down after finitely many iterations,
after which a first-order condition is achieved when restricting to
the fixed active set. A fixed active set here corresponds to
a fixed sequence of heights that are taken by the resulting control
along the interval $(t_0,t_f)$. While the sequence
of heights is fixed, the only quantities that remain to
be optimized are the exact locations at which the height changes.
In this regard, the problem \eqref{eq:sl} is a subproblem that can be
solved for a fixed active set, that is, a fixed switching order.
If \eqref{eq:sl} can be tackled efficiently with
the methods of nonlinear optimization, it can be
invoked to ensure that the inner
loop terminates finitely.
\end{remark}

\begin{remark}
After the the total variation has converged, the switching location
optimization is closely related to \emph{bang-bang}
control, for which a vast literature exists, in particular  on sufficient conditions for bang-bang optimal controls
\cite{maurer2004second,casas2012second}.
Algorithm \ref{alg:trm} may be augmented by using recently developed
techniques for switching point optimization (see, e.g., 
\cite{flasskamp2013discretized,stellato2016optimal,ruffler2016optimal,de2019mixed,de2020sparse})
to improve convergence.
In order to determine small modifications that improve the switching structure,
the \emph{mode insertion gradient} as proposed in \cite[Section 4]{ruffler2016optimal} may 
be a beneficial option to replace some of the, presumably more expensive, trust-region subproblem solves.
\end{remark}

\subsection{Approximation of the Trust-Region Subproblem}\label{sec:tr_trip_approx}
For implementations of Algorithm \ref{alg:trm}, it is important that
the trust-region subproblems (TR) can be approximated with (TRIP) by 
choosing a suitable partition $\mathcal{T}$ of the domain $\Omega$.
While a thorough analysis of this relationship and suitable refinement 
strategies that yield approximation results for discretized versions
of Algorithm \ref{alg:trm} are beyond the scope of this article,
we show an approximation result for one-dimensional domains and a class
of uniformly refined equidistant grids that lead to a practical
implementation of Algorithm \ref{alg:trm}.

With a slight abuse of notation, we consider $\Omega \coloneqq [0,T)$
for simplicity of the presentation.
Let $\tilde{g} \in L^\infty([0,T))$, $\Delta > 0$, and let $\tilde{v}
\in \F_{\eqref{eq:p}} \cap BV([0,T))$ be given.
We consider a fixed
uniform discretization of $[0,T)$ for $N \in \mathbb{N}$ with $h = N^{-1}$, $N = 2^k$
for some $k \in \mathbb{N}$ and $t_i = i h$ for $i \in \{0,\ldots,N\}$. 
We consider the relationship between the optimization problems
\begin{gather}\label{eq:tr_var}
\begin{aligned}
\min_{v}\ &\Upsilon(v) \coloneqq \int_0^T \tilde{g}(s)v(s)\,\dd s + \alpha\TV(v) \\
\text{ s.t.\ }
&\|v - \tilde{v}\|_{L^1([0,T))} \le \Delta,\\
&v(s) \in V \text{ for a.a.\ } s \in [0,T)
\end{aligned}
\end{gather}
and
\begin{gather}\label{eq:trip_var}
\begin{aligned}
\min_{v,b_1,\ldots,b_N}\ &\Upsilon(v)\\
\text{s.t.\ }
& \|v - \tilde{v}\|_{L^1([0,T))} \le \Delta,\\
& v(s) = \sum_{i=1}^N b_i \chi_{[t_{i-1}, t_i)}(s) \text{ for a.a.\ } s \in [0,T) \text{ with } b_1,\ldots,b_N \in V.
\end{aligned}
\end{gather}
One can easily  see that \eqref{eq:tr_var} is equivalent to 
$(\text{TR}(\tilde{v},\tilde{g},\Delta))$ by adding an
offset for the two constant terms that  depend only on $\tilde{v}$. 
Moreover, \eqref{eq:trip_var} is the instance of \eqref{eq:trip}
derived from $(\text{TR}(\tilde{v},\tilde{g},\Delta))$
by choosing the uniform discretization introduced above and the
constants as described in section \ref{sec:trip}.
The feasible set of \eqref{eq:trip_var} is a subset of the feasible set of
\eqref{eq:tr_var}, and thus
\begin{gather}\label{eq:min_trv_tripv_estimate}
\min {\eqref{eq:tr_var}} \coloneqq \min \{ \Upsilon(v) \,|\, v \text{ feasible for } \eqref{eq:tr_var} \}
\le \min \{ \Upsilon(v) \,|\, v \text{ feasible for } \eqref{eq:trip_var}\}
\eqqcolon \min {\eqref{eq:trip_var}}.
\end{gather}

We consider the following setting. We assume that $\tilde{v}$ is a 
piecewise constant function that is (also) defined on an equidistant 
partition into $2^{k_0}$ intervals for $k_0 \in \mathbb{N}$.
This assumption can (inductively) be satisfied in implementations
of Algorithm \ref{alg:trm} if $v^0$ is defined on such a partition
and all trust-region subproblems are replaced by IPs of the form \eqref{eq:trip_var}. 
Moreover, let $v$ be a minimizer of \eqref{eq:tr_var}.
Then, there exists $n \in \mathbb{N}$ such that
\begin{gather}\label{eq:disc_minimizer}
v = \sum_{i=1}^n a_i \chi_{[s_{i-1}, s_i)}
\end{gather}
for $0 = s_0 < \ldots < s_n = T$ and $a_1,\ldots,a_n \in V$ with $a_i \neq a_{i+1}$
for all $i \in \{1,\ldots,n\}$ by virtue of Proposition \ref{prp:vta_representation}.

We consider the setting given by $N$, $h$, $\tilde{g}$, $\tilde{v}$,
$\Delta$, the uniform discretization, \eqref{eq:tr_var}, and 
\eqref{eq:trip_var} described above. In this setting, we show
the following proposition.
\begin{proposition}\label{prp:approx_relationship_tr_trip}
Let $\tilde{g} \in L^\infty([0,T))$, $\Delta > 0$, and
$\tilde{v} \in \F_{\eqref{eq:p}} \cap BV([0,T))$. Let
$\tilde{v}$ be a piecewise constant function that is defined on an
equidistant partition into $2^{k_0}$ intervals for some $k_0 \in \mathbb{N}$.
Let $v$ in \eqref{eq:disc_minimizer} be a minimizer of \eqref{eq:tr_var}.
There exist $N_0 \in \mathbb{N}$ such that if $N \ge N_0$,
then there exist $(v^h,b_1^h,\ldots,b_N^h) \in BV([0,T)) \times V^N$ 
feasible for \eqref{eq:trip_var} such that
\begin{gather}\label{eq:approximation_tr_trip_estimate}
\min {\eqref{eq:trip_var}} \le
\min {\eqref{eq:tr_var}}
+ 2 h \|\tilde{g}\|_{L^\infty([0,T))} \sum_{i=1}^n |a_i|.
\end{gather}
\end{proposition}
\begin{proof}
We choose $N_0 \ge 2^{k_0}$ large enough such that for all $i \in \{1,\ldots,n\}$
there exists $j \in \{1,\ldots,N_0\}$ with $s_{i-1} < t_{j} < s_{i}$.

We define an approximation of $v$ on the equidistant grid $t_0,\ldots,t_N$
as follows.
For $i \in \{0,\ldots,n\}$, we define the indices $j(i)$ iteratively.
We set $j(0) = 0$. For $i > 1$, let $j_\ell(i) \coloneqq \max\{j \in \{1,\ldots,N\}\,|\,t_j \le s_i\}$
and $j_u(i) \coloneqq \min\{j \in \{1,\ldots,N\}\,|\,s_i < t_j\}$.
We choose $j(i) = j_\ell(i)$ if $i + 1 \le n$ and
\begin{multline*}
\int_{0}^{t_{j_u(i)}}
   \left|\sum_{k=1}^{i-1} a_k \chi_{[t_{j(k - 1)}, t_{j(k)})}(s)
   + a_i \chi_{[t_{j(i - 1)}, t_{j_\ell(i)})}
   + a_{i+1} \chi_{[t_{j_\ell(i)}, t_{j_u(i)})}
   - \tilde{v}(s)\right|\dd s\\
   \le 
   \int_{0}^{t_{j_u(i)}}
   \left|\sum_{k=1}^{i-1} a_k \chi_{[t_{j(k - 1)}, t_{j(k)})}(s)
   + a_i \chi_{[t_{j(i - 1)}, t_{j_u(i)})}
   - \tilde{v}(s)\right|\dd s
\end{multline*}
hold; otherwise we choose $j(i) = j_u(i)$. In particular, it holds
that $j(n) = N$. We define
\[ v^h \coloneqq \sum_{i=1}^n a_i \chi_{T_i}
\text{ with }
T_i \coloneqq [t_{j(i-1)},t_{j(i)}) \text{ for all }
i \in \{1,\ldots,n\}.
\]
We note that $v^h$ is constructed on a refinement of the partition
on which $\tilde{v}$ is defined. Combining this with the 
iterative construction of the $j(i)$ and hence the $T_i$, it follows
inductively that
\[ \|\tilde{v} - v^h\|_{L^1([0,t_{j(i)})} \le \|\tilde{v} - v\|_{L^1([0,t_{j(i)})}
\]
for all $i \in \{1,\ldots,n\}$.
This proves $\|\tilde{v} - v^h\|_{L^1([0,T))} \le\Delta$
and hence the feasibility of $v^h$.

Let $d_{\text{lin}} \coloneqq \Upsilon(v^h) + \alpha \TV(v^h)  - (\Upsilon(v) + \alpha \TV(v))$.
Then, we estimate
\begin{align*}
d_{\text{lin}}
&\le \left|\int_0^Tg(s)\sum_{i=1}^n a_i(\chi_{T_i}(s) - \chi_{[s_{i-1},s_i)}(s))\,\dd s\right|\\
&\le \|g\|_{L^\infty((0,T))} \sum_{i=1}^n |a_i|
  \int_0^T \left|\chi_{T_i}(s) - \chi_{[s_{i-1},s_i)}(s)\right|\,\dd s\\
&\le \|g\|_{L^\infty((0,T))} \sum_{i=1}^n |a_i| \lambda(T_i \Delta [s_{i-1},s_i))\\
&\le 2 h \|g\|_{L^\infty((0,T))} \sum_{i=1}^n |a_i|,
\end{align*}
where $T_i \Delta [s_{i-1},s_i)$ denotes the symmetric difference
between the sets $T_i$ and $[s_{i-1},s_i)$. Moreover,
the estimate $\lambda(T_i \Delta [s_{i-1},s_i)) \le 2h$ in the
last inequality holds because the construction of $j(i)$
implies $|t_{j(i)} - s_i| \le h$ for
all $ i\in \{0,\ldots,n\}$.
Moreover, the construction of $v^h$ implies
that $\alpha \TV(v^h) \le \alpha \TV(v)$.

We combine the estimate on $d_{\text{lin}}$ and
the inequality $\alpha \TV(v^h) \le \alpha \TV(v)$ with the
inequality \eqref{eq:min_trv_tripv_estimate}
to obtain the estimate \eqref{eq:approximation_tr_trip_estimate},
which closes the proof.
\end{proof}

Because of the dependence of $N_0$ in
Proposition \ref{prp:approx_relationship_tr_trip} on $2^{k_0}$
it is  conceivable that very fine grids are necessary in an
implementation with approximation error control. More work is necessary to
develop more sophisticated approximations, also in the one-dimensional
case. 
Due to these considerations, we solve instances
of the IP formulation of \eqref{eq:trip_var} for uniform
discretizations for different values of $N$ in our computational 
experiments below to demonstrate the efficacy of the proposed idea.

\section{Computational Experiments}\label{sec:experiments}

To assess our algorithm computationally,
we use the following optimization problem.
\begin{gather}\label{eq:ex_p}
 \min \frac{1}{2}\|K v - f\|_{L^2((t_0,t_f))}^2 + \alpha\TV(v) \text{ s.t. } v(t) \in \{-2,-1,0,1,2\} \text{ a.e.}
\end{gather} 
Here, $t_f - t_0 = 2$, $f \in L^2((t_0,t_f))$ is given,
and we use $Kv \coloneqq k * v$ for a fixed convolution
kernel $k \in L^{\infty}((0,2))$,
where $*$ denotes the
convolution operator defined by $(k * v)(t) = \int_{t_0}^t k(t - \tau)v(\tau)\,\dd\tau$ for $t \in [t_0,t_f]$.
Regarding the data we choose $f(t) = 0.4\cos(2\pi t)$ for $t \in (-1,1)$
as well as 
\[ k(t) = -0.1\omega_0
\left(\exp\left(-\frac{\omega_0(t - 1)}{\sqrt{2}}\right)\cos\left(\frac{\omega_0(t - 1)}{\sqrt{2}}
- \frac{\pi}{4}\right)
+ \exp\left(-\frac{\omega_0(t - 1)}{\sqrt{2}}\right)\sin\left(\frac{\omega_0(t - 1)}{\sqrt{2}}- \frac{\pi}{4}\right)\right)\]
for $t \in (0,2)$ and $\omega_0 = \pi$. 
Note that the same kernel function is used but reported incorrectly in \cite{kirches2020compactness,manns2020relaxed}.

We perform three experiments. In two of them, we choose $\alpha = 0.0001$ in order to obtain a
meaningful comparison with the combinatorial integral approximation decomposition, which is not possible for high
values of $\alpha$ because the objective becomes \emph{reduce switching by all costs} in this case,
which is fundamentally opposed to the combinatorial integral approximation approach.
In general, $\alpha$ allows one to control the trade-off between control switching and the
minimization of the (tracking) term $F$, analogously to regularization of inverse problems
with fractional-valued control inputs. We highlight that some care is necessary
here because choosing $\alpha$ too large may imply that the resulting control is constant.
We leave a rigorous strategy how to determine $\alpha$ to future work. In order to shed some light
on this situation, we perform a third computational experiment, where we vary the regularization
parameter $\alpha$ for a set of randomly drawn initial values of our implementation of the SLIP method.

We briefly verify that \eqref{eq:ex_p} satisfies our assumptions.

\begin{proposition}
Problem \eqref{eq:ex_p} satisfies
Assumptions \ref{ass:f_setting} and \ref{ass:d1}.
\end{proposition}
\begin{proof}
We define $\Omega \coloneqq (t_0,t_f)$ and
$F(v) \coloneqq \frac{1}{2}\|Kv - f\|_{L^2(\Omega)}^2$
for $v \in L^2(\Omega)$. $F$ is bounded from below by zero.
Because the squared norm $\|\cdot\|_{L^2(\Omega)}^2$ and the operator $K$ are
continuous (see \cite[Thm.\ 3.1.17]{simon2015operator}),
$F$ is lower semi-continuous, and Assumption \ref{ass:f_setting}
holds.

A straightforward calculation shows that $F$ is continuously differentiable.
Because of the Hilbert space setting, the gradient
$\nabla F : L^2(\Omega) \to L^2(\Omega)$ is available and 
is given as
\[ \nabla F(v) = K^*(K v - f), \]
where $K^*$ denotes the adjoint operator of $K$.
Because $K$ is linear and bounded, this also holds for $K^*$;
and consequently $\nabla F$ is a continuously differentiable affine function.

Young's convolution inequality, $k \in W^{1,\infty}_c(\R)$,
and the continuous embedding $W^{1,\infty}(\R) \hookrightarrow L^2(\R)$
give the continuity of the operator $K : L^1(\Omega) \to L^2(\Omega)$.
Moreover, we obtain
\[ \left|\nabla^2 F(\xi)(u,w)\right| = \left|(K^*Ku,w)_{L^2(\Omega)}\right|
   = \left|(Ku,Kw)_{L^2(\Omega)}\right|
   \le \|Ku\|_{L^2(\Omega)}\|Kw\|_{L^2(\Omega)}   
\]
for all $\xi \in L^2(\Omega)$. Combining these observations with
Proposition \ref{prp:d1_satisfied} yields that Assumption \ref{ass:d1}
is satisfied.
\end{proof}

An unregularized variant of \eqref{eq:ex_p} has been solved
after discretization to global optimality in
\cite{buchheim2012effective}. The authors 
show that global optimization
becomes computationally intractable for fine discretizations,
in their case over $\approx 120$ intervals.
An unregularized variant of \eqref{eq:ex_p}
has been approached with
the combinatorial integral approximation decomposition
in \cite{kirches2020compactness}, where high-frequency
switching occurs for fine discretization grids.
In \cite{manns2020relaxed} a different regularization
is chosen for \eqref{eq:ex_p}, namely, a convex
relaxation of the multibang regularizer introduced
in \cite{clason2014multi}. The combination of the relaxed
multibang regularization with approximation algorithms
that take switching costs (see \cite{bestehorn2020mixed})
into account in the combinatorial integral approximation can
yield a reduced number of switches in \cite{manns2020relaxed}.

We have run three experiments. In \emph{Experiment 1} with results
described in section \ref{sec:results_experiment_1}, we compare
four variants of the SLIP method for a sequence of refined discretizations.
Specifically, we compare the SLIP method's efficiency with and without 
initialization with the solution from the previous coarser
discretization grid and for two different update strategies
for the trust-region radius.

In \emph{Experiment 2}, described in section \ref{sec:results_experiment_2},
we compare the SLIP method with a MINLP approach of
solving \eqref{eq:p} for a sequence of refined discretization grids.
We also point out briefly the difference from the combinatorial integral
approximation decomposition approach.
For this experiment, we choose the SLIP variant from \emph{Experiment 1}
that gave the best results in terms of the objective
on the finest discretization grid.

In \emph{Experiment 3}, described in section \ref{sec:results_experiment_3},
we assess the sensitivity of the solution obtained with the SLIP method with respect
to different initialization points and regularization parameters $\alpha$.

We use the SCIP Optimization Suite (version 7.0) \cite{GamrathEtal2020ZR}
with SoPlex as the linear programming solver to solve both the 
\eqref{eq:trip} subproblems of the SLIP method and the MINLPs, which are
mixed-integer quadratic programs (MIQPs).
We note that we have not experienced a qualitative difference
in the results when choosing a different IP solver.
Both algorithms are executed on a compute server with
two Intel(R) Xeon Gold 6130 CPUs (16 cores each) clocked
at 2.1 GHz and with 192 GB RAM.

\subsection{Discretization}\label{sec:slip_discretization}

The SLIP method is a function space algorithm and therefore cannot be
implemented directly on a computer. Specifically, the solution of the
trust-region subproblem \eqref{eq:tr} Line \ref{ln:trstep} cannot be
done exactly.
Following the considerations in sections \ref{sec:trip}
and \ref{sec:tr_trip_approx}, we approximate \eqref{eq:tr}
as follows. For $N \in \{32,64,128,256,512,1024,2048\}$, we
define a finest possible equidistant control discretization
into $N$ intervals. Thus, for a given $N$, a control function
is an interval-wise constant function on $N$ intervals that
equidistantly discretize $(t_0,t_f)$.

Because $\TV(v)$ is integral for integer-valued
control functions $v$, the
term $\alpha\TV$ can be implemented with machine precision accuracy.
Additionally, terms of the forms
\[ F(v) = \frac{1}{2}\|Kv - f\|_{L^2((t_0,t_f))}^2
   = \int_{t_0}^{t_f} |(k * v)(t) - f|^2 \,\dd t, \]
and
\[ (\nabla F(\tilde{v}),v)_{L^2(\Omega)} = (K\tilde{v} - f,Kv)_{L^2(\Omega)}
                        = \int_{\Omega}((k * \tilde{v})(t) - f(t))(k * v)(t)
                        \,\dd t \]
are required to evaluate the objectives of \eqref{eq:p} and \eqref{eq:tr}
for control functions $v$ and $\tilde{v}$ that are interval-wise constant.
The accuracy of the evaluation of the integrals on the right-hand sides
of these equations depends on the accuracy of the evaluation of the convolutions
$k * v$ and $k * \tilde{v}$.
We approximate them with fifth-order Legendre--Gauss quadrature rules
per interval for a decomposition into $2,048$ intervals.

All integrals are using the decomposition into $2,048$ intervals. For smaller
values of $N$, the control functions are broadcast to functions defined
on $2,048$ intervals in our implementation. This procedure yields that the
objectives are consistent over the different discretizations, which
allows a comparison of the achieved objective values.

We assemble all of these approximations in the finite-dimensional
IP \eqref{eq:trip}, which we then solve with a general-purpose IP solver.

Similarly, we derive corresponding mixed-integer quadratic programs
(MIQPs) from \eqref{eq:ex_p} for the same integral discretizations
and the same choices of $N$. In the cases where we can solve
an MIQP to global optimality, the resulting optimal objective value 
is a lower bound on the result of the SLIP method.

Each value of $N$ yields a smallest possible trust-region radius
$\underline{\Delta}^N \coloneqq (t_f - t_0)/N$.
Thus, for given initial trust-region radius $\Delta^0$ and the choice of $N$,
our implementation allows for at most
$k_{\max}^N = \log_2(\underline{\Delta}^N/\Delta^0)$ consecutive
refinements of the trust-region radius.
Our implementation of the SLIP method stops if the sufficient decrease
condition \eqref{eq:1d_accept} cannot be satisfied with the trust-region
radius $\underline{\Delta}^N$ in the inner loop of the SLIP method.

\subsection{Comparison of SLIP Configurations}\label{sec:results_experiment_1}

We compare four variants of the SLIP method.
Specifically, we run the SLIP method with two different trust-region
update strategies and two different ways of initializing the algorithm.

One trust-region update strategy is the trust-region reset strategy (R)
that is defined in Algorithm \ref{alg:trm} and uses the
reset trust-region radius $\Delta^0 = 0.125$.
The other trust-region update strategy (D) works as follows. If a step
is accepted, we double the trust-region radius for the next iteration.

One initialization strategy is the use of the control $v^0 \equiv 0$
for all rounding grids (0). The other initialization strategy
is to use $v^0 \equiv 0$ for only $N = 32$ and use the resulting
control from the optimization with $N / 2$ control intervals
as initialization for the optimization on $N$ control intervals (P).
This strategy corresponds to mesh sequencing.

In total we obtain four algorithm variants, which we denote by
SLIP R 0, SLIP R P, SLIP D 0, and SLIP D P.
Because the mesh-sequencing strategies SLIP R P and SLIP D P use the
results from the previous smaller values of $N$, we accumulate their
compute time over the values $32,\ldots,N$ intervals for
the run with a control discretization into $N$ intervals.

The  objective values achieved are similar for all four algorithm
variants. Regarding the final objective value on the finest
grid, SLIP R 0 performs slightly better than SLIP R P and
SLIP D P, which perform similar to and slightly better than SLIP D 0, respectively.
For the compute times, the differences are much larger. Specifically,
regarding the compute times for the final objective on the
finest grid, SLIP R 0 takes considerably longer than all the other
variants with $1.70\cdot 10^4\,s$. SLIP D P and SLIP R P take
$2.31\cdot 10^3\,s$ and $2.88\cdot 10^3\,s$, respectively,  less
than a fifth of the compute time of SLIP R 0.
The fastest method is SLIP D 0, which  takes only $6.16\cdot 10^2\,s$.
We have tabulated the objective values and compute times
for all variants of SLIP and all $N$ in Table
\ref{tbl:slip_strategies}.
SLIP D 0 is the only strategy in which the optimization on
the finest control grid with $N = 2048$ does not take much
longer than the computation of the grids until $N = 1024$.
The best values achieved in terms of objective and compute
time are highlighted in boldfaced font.

\begin{table}[t]
\caption{Best achieved objectives and corresponding
times of the runs of the different variants
of the SLIP method. Best values (lowest objective and lowest
time) are highlighted using boldfaced font.}\label{tbl:slip_strategies}
\begin{center}
\begin{tabular}{rllllllll}
\hline
& \multicolumn{2}{c}{SLIP R 0}
& \multicolumn{2}{c}{SLIP R P}
& \multicolumn{2}{c}{SLIP D 0}
& \multicolumn{2}{c}{SLIP D P} \\
N
& objective
& time [$s$]
& objective
& time [$s$]
& objective
& time [$s$]
& objective
& time [$s$]\\ \hline
32   & $\pmb{9.08\cdot 10^{-3}}$
     & $\pmb{0.354}$ 
     & $\pmb{9.08\cdot 10^{-3}}$
     & $\pmb{0.354}$ 
     & $9.96\cdot 10^{-3}$
     & $0.680$
     & $9.96\cdot 10^{-3}$
     & $0.680$\\     
64   & $9.17\cdot 10^{-3}$  
     & $\pmb{1.02}$
     & $\pmb{6.17\cdot 10^{-3}}$
     & $1.50$
     & $7.74\cdot 10^{-3}$
     & $1.52$
     & $7.16\cdot 10^{-3}$
     & $1.25$\\
128  & $7.08\cdot 10^{-3}$  
     & $3.19$
     & $\pmb{5.66\cdot 10^{-3}}$
     & $5.66$
     & $6.93\cdot 10^{-3}$ 
     & $3.86$
     & $6.66\cdot 10^{-3}$
     & $\pmb{2.98}$\\
256  & $5.52\cdot 10^{-3}$
     & $\pmb{2.35\cdot 10^{1}}$
     & $\pmb{5.00\cdot 10^{-3}}$ 
     & $3.66\cdot 10^{1}$
     & $5.19\cdot 10^{-3}$
     & $2.61\cdot 10^{1}$
     & $5.22\cdot 10^{-3}$
     & $2.40\cdot 10^{1}$\\
512  & $\pmb{4.43\cdot 10^{-3}}$
     & $1.69\cdot 10^{2}$
     & $4.52\cdot 10^{-3}$
     & $1.70\cdot 10^2$
     & $4.44\cdot 10^{-3}$
     & $1.31\cdot 10^{2}$
     & $5.05\cdot 10^{-3}$
     & $\pmb{8.67\cdot 10^1}$ \\
1024 & $4.53\cdot 10^{-3}$
     & $3.30\cdot 10^{2}$ 
     & $\pmb{4.49\cdot 10^{-3}}$
     & $3.07\cdot 10^{2}$
     & $4.77\cdot 10^{-3}$
     & $\pmb{2.52\cdot 10^{2}}$
     & $4.63\cdot 10^{-3}$
     & $6.63\cdot 10^2$ \\     
2048 & $\pmb{4.34\cdot 10^{-3}}$
     & $1.70\cdot 10^{4}$
     & $4.49\cdot 10^{-3}$ 
     & $2.88\cdot 10^{3}$ 
     & $4.71\cdot 10^{-3}$ 
     & $\pmb{6.16\cdot 10^2}$
     & $4.45\cdot 10^{-3}$
     & $2.31\cdot 10^3$ \\ \hline
\end{tabular}
\end{center}
\end{table}

\subsection{Comparison of SLIP and MINLP Solvers}\label{sec:results_experiment_2}

In this experiment, both the SLIP method and the MIQP solves
are initialized with the initial control $v^0 \equiv 0$
for all control discretizations. We used
the same reset trust-region radius $\Delta^0 = 0.125$
for all control discretizations.
This corresponds to variant SLIP R 0 in the preceding section, which
showed the best performance in terms of the objective value on the finest
grid.

The longest runtime of the SLIP method was 16,985\,s for $N = 2,048$.
The MIQP runs were stopped if they did not complete after
five hours (18,000\,s).
The MIQP run for $N = 32$ was able to find and certify a global minimum
after 5670\,s. The other MIQP runs did not solve \eqref{eq:ex_p} to
global optimality within the time limit, with smallest reported duality
gaps of 100\,\% and more.

For $N = 32$, $64$, and $128$, the MIQP solves with SCIP are able to
produce better objective values than the SLIP method, although
their compute time is much longer.
For example, the best objective achieved for $N = 64$ is $4.733\cdot 10^{-3}$,
which is produced by SCIP for the MIQP formulation after 4,604\,s.
The objective achieved by the SLIP method is $9.169\cdot 10^{-3}$, but the
SLIP method  requires only $4\,s$.
In
Table \ref{tbl:best_achieved} we tabulate the best objectives achieved
by the SLIP method and the MIQP formulation as well as the runtime
of the SLIP method and the time when the solution with the best
objective is produced for the MIQP formulation.
For each $N$ the best achieved objective value is highlighted with
boldfaced font.

For $N = 32$, $64$, and $128$, the SLIP method produces its final
iterate before a better objective is produced by the MIQP solves.
In the other runs, the SLIP method produces better objectives in
much shorter time horizons than do the MIQP solves.
The best objective value over all runs and methods is produced by the
SLIP method for $N = 2,048$ after 16,804\,s.
To understand these results better, we have visualized in Figure 
\ref{fig:obj_over_time} the trajectories of the
achieved objectives over (compute) time for both methods and all grids.

To assess the difference from the combinatorial integral approximation
decomposition approach, we have removed the total variation term from the 
objective and run the combinatorial integral approximation decomposition
on \eqref{eq:ex_p}. This means that we have solved the continuous 
relaxation of the control problem \eqref{eq:phat} and then computed a rounding of
the resulting fractional-valued controls to $\{-2,-1,0,1,2\}$-valued 
controls using the SCARP approach 
\cite{bestehorn2019switching,bestehorn2020mixed}.
We have solved the continuous relaxation and performed the rounding
for the same control discretizations.
When SCARP is used in the rounding step, the total variation of the resulting
control is minimal for a given distance to the fractional-valued
relaxed control (the prescribed distance is linear in the interval 
length).
The objective values for \eqref{eq:ex_p} of the discrete-valued
controls produced by SCARP are given in the sixth column of
Table \ref{tbl:best_achieved}. Most of the time for
the combinatorial integral approximation decomposition approach
is spent for the solution of the continuous relaxation, which
is computed on the finest grid and independently of the
rounding grids that are used for SCARP, which is very fast.
Therefore, the timings in the last column of Table \ref{tbl:best_achieved}
do not differ much.

For the coarsest grid, the combinatorial integral approximation
decomposition approach produces less switching
than does the SLIP method,  while the tracking term of the objective
is worse. As guaranteed by the theory behind the combinatorial
integral decomposition (see, for example,
\cite{kirches2020compactness,manns2018multidimensional}),
the value of the tracking term tends to the value of the
continuous relaxation with finer control discretizations.
This comes at a cost of increased switching behavior; the
total variation term in the objective tends to infinity
for $N \to \infty$. For our set of parameters, we  observe
that the trade-off produced by the combinatorial integral
approximation yields a superior result in terms
of the objective value for $N = 64$ and $N = 128$ compared
with the SLIP method.
For higher values of $N$ the increase of the switching implies that
the total variation term dominates the tracking term,
and an approximately linear increase of the objective value can
be observed in the results.

\begin{table}[t]
\caption{Best achieved objectives and corresponding
times of the runs of SLIP and the MIQP solves.
For SLIP the final time is reported; for MIQP the time of the computation
of the incumbent solution with the best reported objective is reported (both in seconds). The last column reports the objective
achieved with the combinatorial integral decomposition using SCARP.
Best values (lowest objective) are highlighted using boldfaced font.}
\label{tbl:best_achieved}
\begin{center}
\begin{tabular}{rllllll}
\hline
& \multicolumn{2}{l}{SLIP}
& \multicolumn{2}{l}{MIQP}
& SCARP\\
N
& objective
& time
& objective
& time to best objective
& objective
& time \\ \hline
32   & $9.081\cdot 10^{-3}$  & $3.543 \cdot 10^{-1}\,s$ 
     & $\pmb{5.079\cdot 10^{-3}}$  & $3.837 \cdot 10^{3}\,s$
     & $1.839\cdot 10^{-2}$ & $1.775\cdot 10^{1}\,s$ \\     
64   & $9.169\cdot 10^{-3}$  & $1.015 \,s$
     & $\pmb{4.733\cdot 10^{-3}}$ & $4.064\cdot 10^{3}\,s^\ast$
     & $6.369 \cdot 10^{-3}$ & $1.775\cdot 10^{1}\,s$ \\     
128  & $7.080\cdot 10^{-3}$  & $3.185\,s$
     & $\pmb{5.447\cdot 10^{-3}}$  & $1.434\cdot 10^{4}\,s^\ast$
     & $5.551 \cdot 10^{-3}$ & $1.778\cdot 10^{1}\,s$ \\
256  & $5.523\cdot 10^{-3}$  & $2.350\cdot 10^{1}\,s$
     & $\pmb{5.513\cdot 10^{-3}}$  & $1.644\cdot 10^4\,s^\ast$
     & $7.741\cdot 10^{-3}$ & $1.777\cdot 10^{1}\,s$ \\     
512  & $\pmb{4.426\cdot 10^{-3}}$  & $1.687\cdot 10^{2}\,s$
     & $6.685\cdot 10^{-3}$  & $1.776\cdot 10^4\,s^\ast$
	 & $1.220\cdot 10^{-2}$ & $1.778\cdot 10^{1}\,s$ \\
1024 & $\pmb{4.529\cdot 10^{-3}}$  & $3.303\cdot 10^{2}\,s$ 
     & $9.153\cdot 10^{-3}$  & $1.680\cdot 10^{4}\,s^\ast$
	 & $2.350\cdot 10^{-2}$ & $1.784\cdot 10^{1}\,s$\\     
2048 & $\pmb{4.339\cdot 10^{-3}}$ & $1.698\cdot 10^{4}\,s$
     & $2.727\cdot 10^{-2}$  & $1.746\cdot 10^{4}\,s^\ast$
     & $4.610\cdot 10^{-2}$  & $1.800\cdot 10^{1}\,s$\\ \hline
\multicolumn{6}{c}{${}^\ast$ Timeout after $1.8\cdot 10^4$\,s}     
\end{tabular}
\end{center}
\end{table}

We have plotted the resulting control trajectories computed with the SLIP method
for $N = 32$ and $N = 2,048$ in the top row of Figure 
\ref{fig:control_trajectories}.
The corresponding state trajectories and the tracked function $f$
are plotted in the top row of Figure \ref{fig:state_trajectories}.
The corresponding trajectories produced by using the
combinatorial integral approximation decomposition with SCARP
for rounding are shown in the bottom rows.

We also evaluate the numerical results with respect to \ref{eq:sl}-stationarity.
We measure \ref{eq:sl}-stationarity by computing $\|(\nabla F(v)(t_i))_i\|$,
where the $t_i$ are the switching locations of $v$, that is, the $t_i$ where
$\lim_{t \uparrow t_i}v(t) \neq \lim_{t \downarrow t_i}v(t)$.
The point evaluation of $\nabla F(v)$ at the switching locations
of $v$ is  meaningful only if $\nabla F(v)$ is a continuous function.
We briefly argue that this is always the case for our example
in Proposition \ref{prp:derivative_of_example}.
For all discretizations, we obtain a downward trend of the
\ref{eq:sl}-stationarity measure after the first few iterations.
This trend stagnates at about $3 \cdot 10^{-6}$
for the finest discretization $N = 2,048$.
The stationarity measure is plotted over the iterations for
all discretizations in Figure \ref{fig:sl_over_iterations}.
We note that with a fixed available finest discretization 
as we have with $N = 2,048$, one cannot in general  
achieve $\|(\nabla F(v)(t_i))_i\| = 0$ even if our computations
were performed in exact arithmetic. 
However, one could fix the order and heights of the 
switches after the final iteration (thereby fixing also the
total variation) and then optimize the switch locations to
minimize $F(v)$, which is a nonconvex nonlinear program.

\begin{figure}
\begin{center}
  \includegraphics{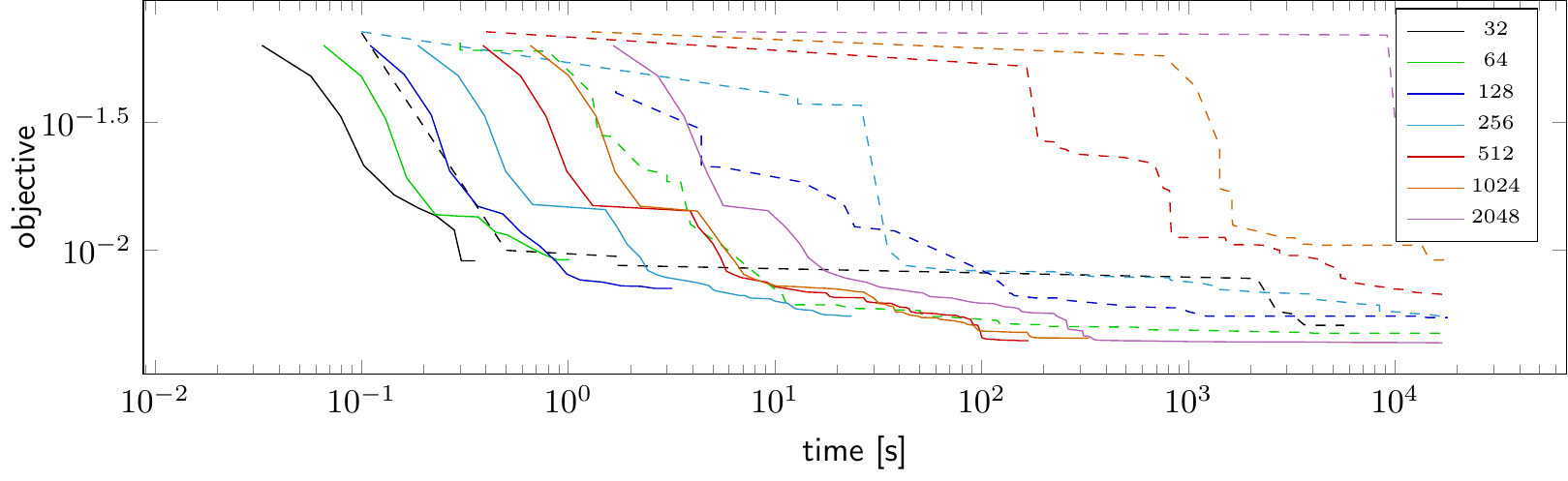}
\end{center}
\caption{Objective values over time for SLIP (solid) and (MIQP) solves  (dashed) for $N = 32$, $\ldots$, $2048$.}\label{fig:obj_over_time}
\end{figure}

\begin{figure}
  \begin{minipage}{.49\linewidth}
    \includegraphics{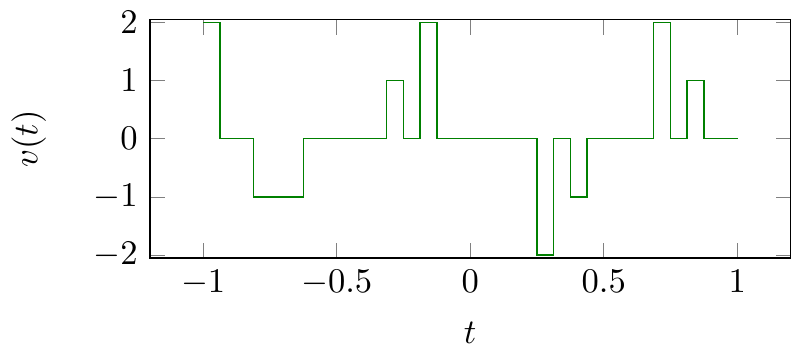}\\
	\includegraphics{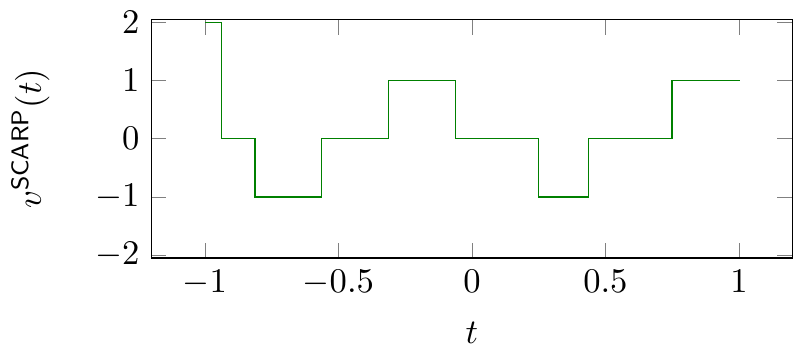}
  \end{minipage}\hfill
  \begin{minipage}{.49\linewidth}
	\includegraphics{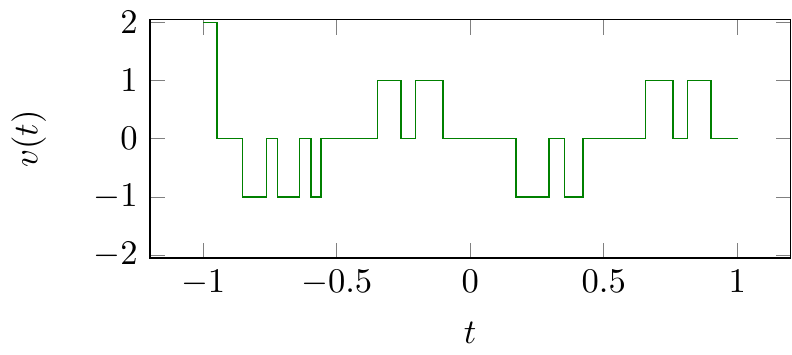}\\
	\includegraphics{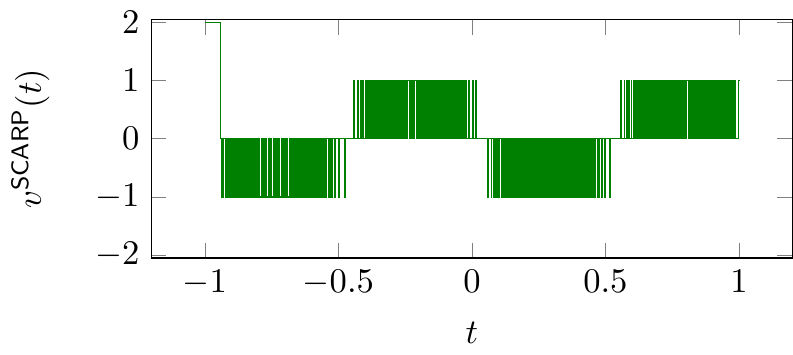}
  \end{minipage}
  \caption{Top row: final control trajectories produced by the SLIP method
  for $N = 32$ (left, objective value $9.081\cdot 10^{-3}$)
  and $N = 2,048$ (right, objective value $4.339\cdot 10^{-3}$).
  Bottom row: control trajectories produced by  the combinatorial
  integral approximation decomposition approach using SCARP for 
  $N = 32$ (left, objective value $1.839\cdot 10^{-2}$) and $N = 2,048$ (right, objective value $4.610\cdot 10^{-2}$).}\label{fig:control_trajectories}
\end{figure}

\begin{figure}
  \begin{minipage}{.49\linewidth}
	\includegraphics{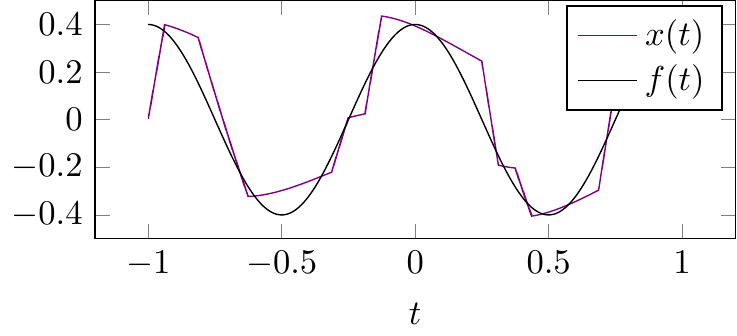}\\
	\includegraphics{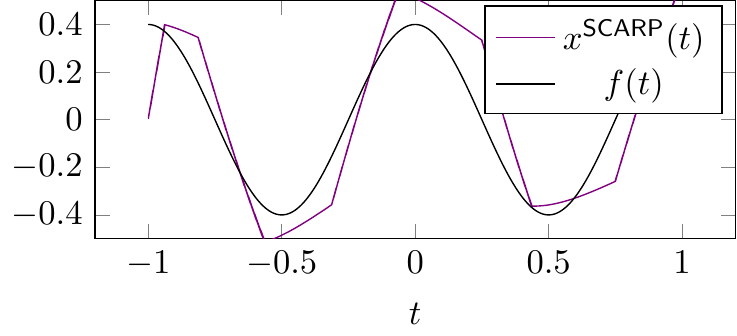}
  \end{minipage}\hfill
  \begin{minipage}{.49\linewidth}
	\includegraphics{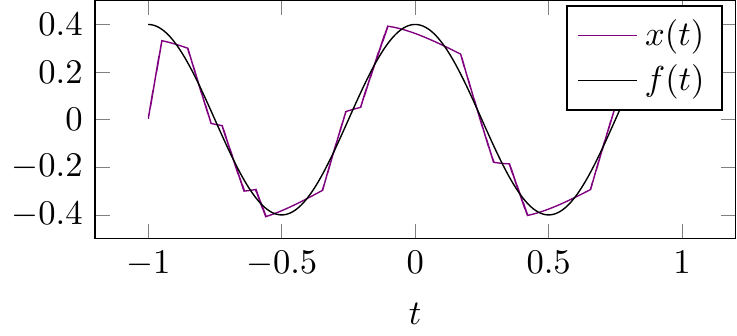}\\
	\includegraphics{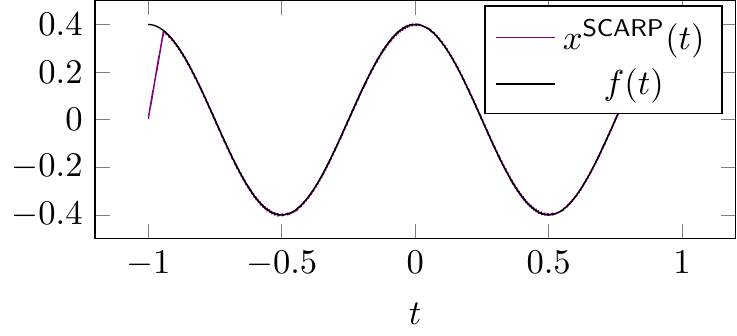}
  \end{minipage}
  \caption{Top row: final state trajectories produced by the SLIP method
  for $N = 32$ (left, objective value $9.081\cdot 10^{-3}$)
  and $N = 2,048$ (right, objective value $4.339\cdot 10^{-3}$)
  in violet and the tracked function $f$ in black.
  Bottom row: final state trajectories produced by the combinatorial
  integral approximation decomposition approach using SCARP for $N = 32$ (left, objective value $1.839\cdot 10^{-2}$) and $N = 2,048$ (right, objective value $4.610\cdot 10^{-2}$). Note that $x(0) = 0$
  and $f(0) = 0.4$ and thus the tracking error is bounded away from zero.}\label{fig:state_trajectories}
\end{figure}

\begin{figure}
\begin{center}
\includegraphics{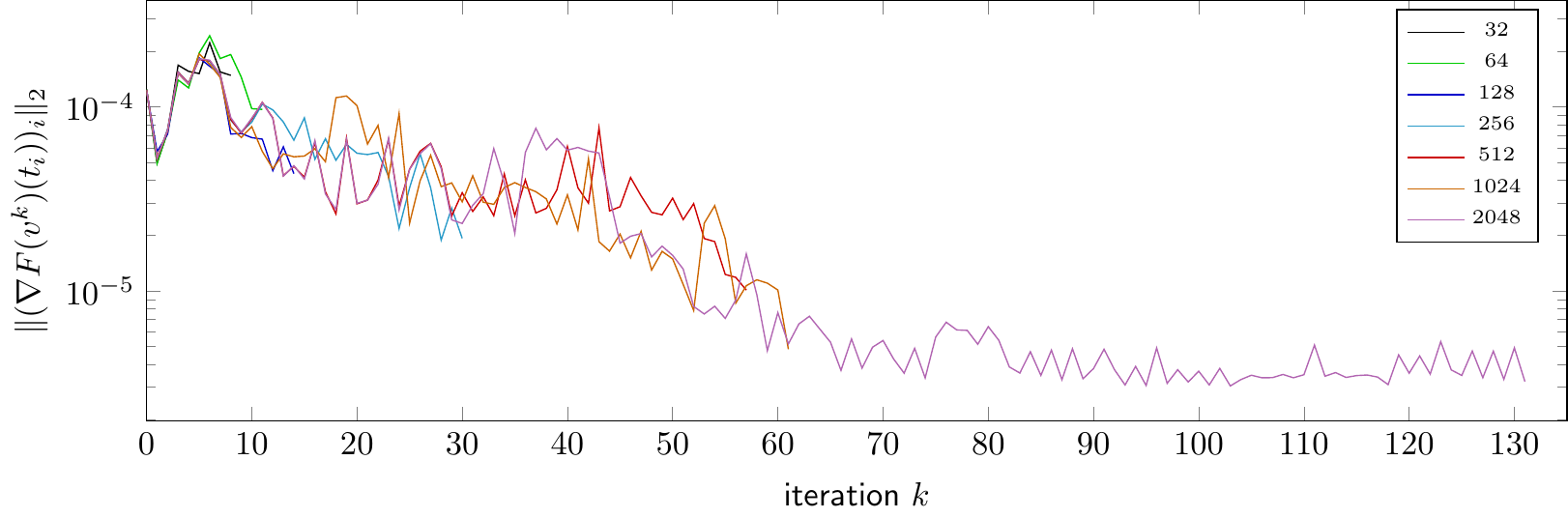}
\end{center}
\caption{\ref{eq:sl}-stationarity of the controls $v^k$
over the iterations $k$ in the SLIP method
for different discretizations.
\ref{eq:sl}-stationarity is measured as the
$\ell_2$-norm of $(\nabla F(v^k)(t_i))_i$ for the switching locations
$t_i$ of $v^k$.}\label{fig:sl_over_iterations}
\end{figure}

\subsection{Sensitivity with Respect to Initialization and Regularization Parameter}\label{sec:results_experiment_3}

We choose $N = 512$ and initialize our implementation of the SLIP method with for 25 randomly generated
controls (drawn from uniform integer distributions of the feasible control values for each  interval)
for the regularization parameter values $\alpha \in \{10^{i}\,|\, i \in \{-5,\ldots,2\} \}$. Again the
SLIP implementation is stopped when the trust-region radius cannot be decreased further.

The mean of the achieved \ref{eq:sl}-stationarity for the final iterate increases for reduced parameters $\alpha$.
Therefore, one may need to choose finer control discretization grids to achieve similar stationarity values for small
regularization parameters. The distributions of the achieved values of stationarity
among the initialization points per regularization parameter value are shown in Figure \ref{fig:sensitivity_sl_obj} (a).

For large values of $\alpha$, the term $\alpha \TV$ dominates the objective and the achieved objective values for
different regularization parameters are largely different while the differences become similar for small regularization
parameters when the first term dominates. In other words, the algorithm is more sensitive to the initialization point
if the regularization parameter is high / switches are penalized hard.

In contrast to \ref{eq:sl}-stationarity, $r$-optimality for some $r > 0$ cannot be assessed directly for the computed
final iterates $v^f$ without solving the infinite-dimensional nonlinear control problem \eqref{eq:p} subject to the
constraint $\|v - v^f\|_{L^1(\Omega)} \le r$. In order to at least estimate the performance in this regard, we
solve the MINLP that has been used in section \ref{sec:results_experiment_2} with the additional constraint
$\|v - v^f\|_{L^1(\Omega)} \le r\frac{t_f - t_0}{N}$, which can be modeled exactly with linear inequalities and
slack variables for our equidistant discretization, for $r \in \{1,2,3,4\}$ and compute the relative differences between
the objective values achieved for $v^f$ and the MINLP solution with the additional constraint.
The relative difference is the difference divided by the objective value of the MINLP solution. The termination criterion
of the MINLP solver is chosen as $10^{-4}$ for the duality gap (difference between upper and lower bound on the MINLP solution
divided by the value of the lower bound). The results show that within this tolerance, the results
produced by our implementation of SLIP are also optimal for the MINLP with the additional constraint
$\|v - v^f\|_{L^1(\Omega)} \le r\frac{t_f - t_0}{N}$ for $\alpha \ge 10^{-2}$. The relative differences increase
with increasing values of $r$, that is a larger feasible set for the MINLP, and decreasing values of $\alpha$.
This is in line with the higher \ref{eq:sl}-stationarity values for decreasing values of $\alpha$.
We provide the achieved relative differences in Figure \ref{fig:sensitivity_gap}. Relative differences
below the duality gap threshold of the MINLP solver are set to the duality gap threshold in the interest
of a fair comparison and a clearer presentation.

\begin{figure}
\begin{center}
  \begin{minipage}{.29\linewidth}
  \includegraphics[height=5cm]{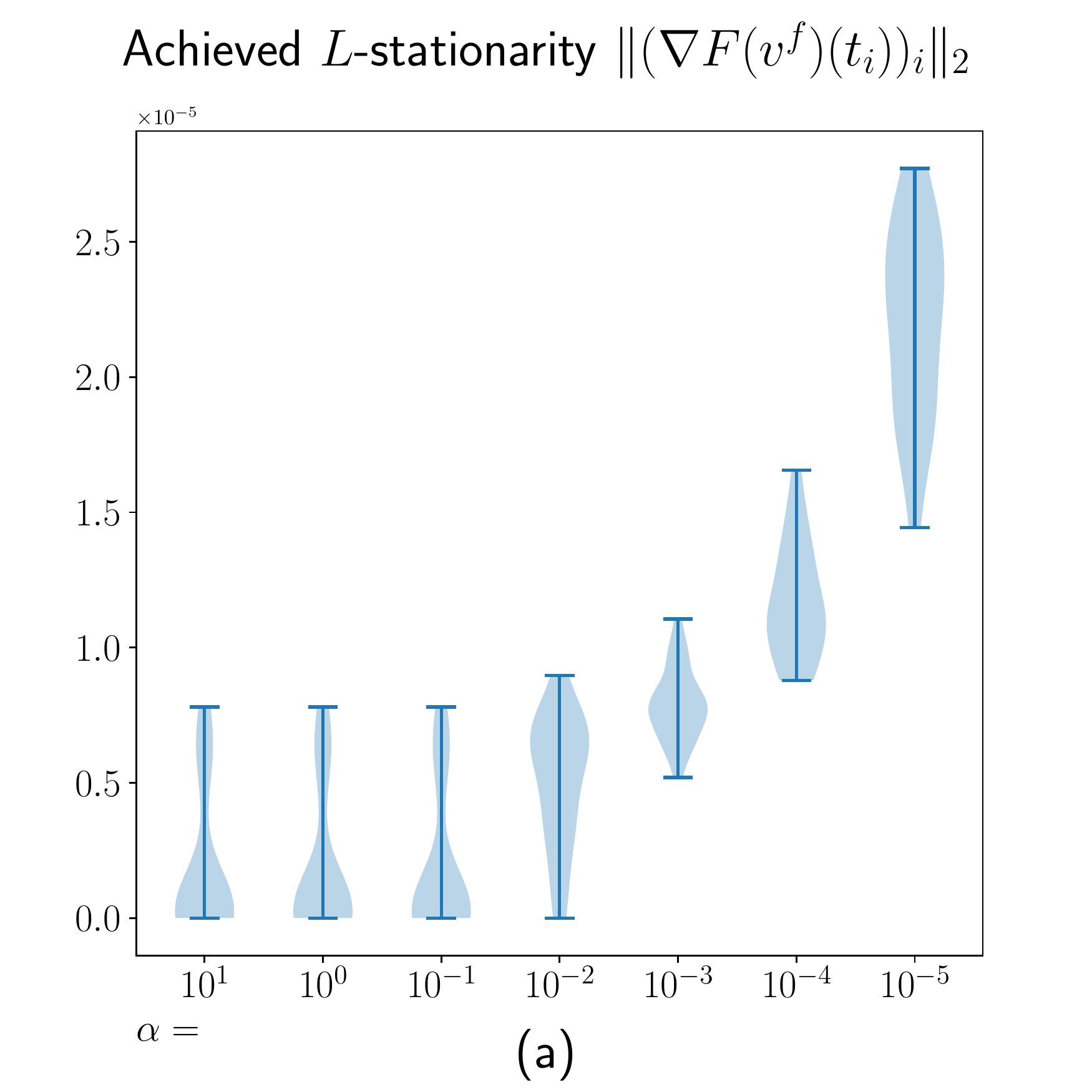}
  \end{minipage}
  \begin{minipage}{.7\linewidth}
    \includegraphics[height=5cm]{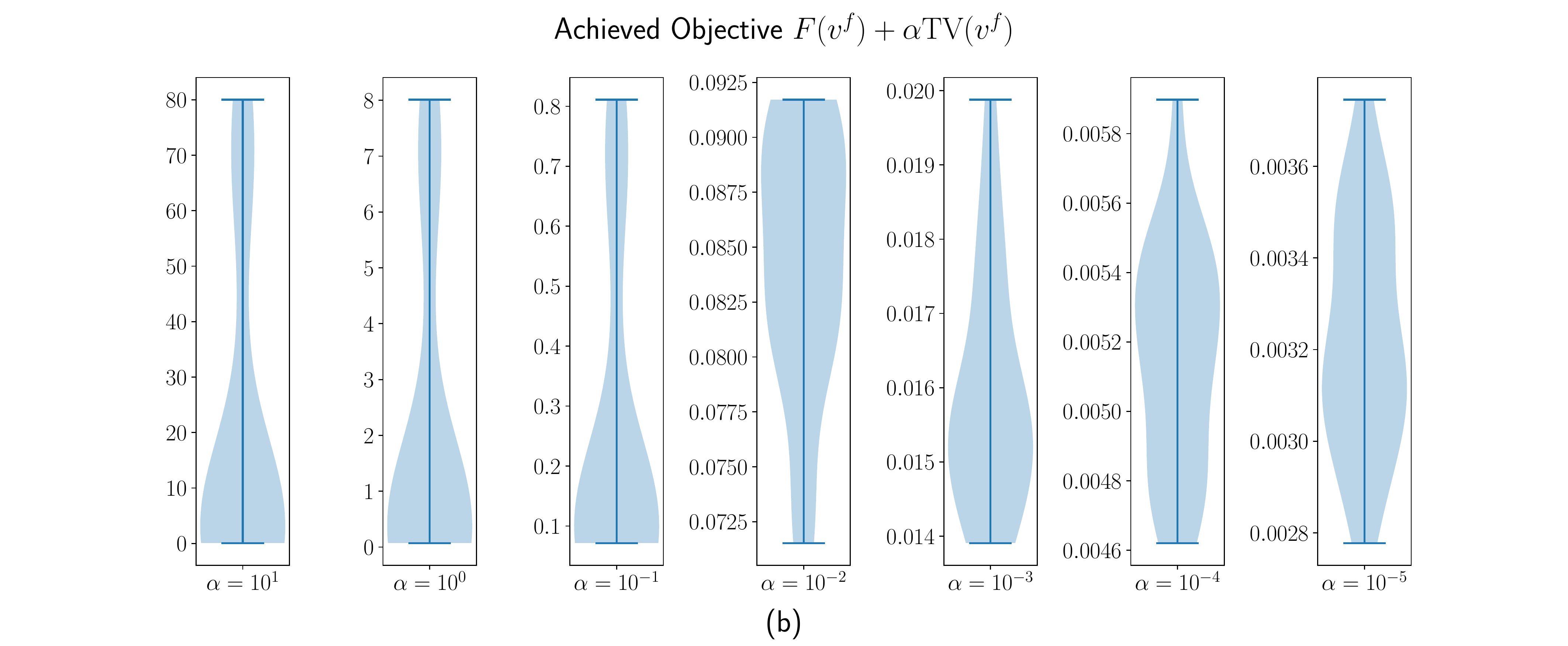}
  \end{minipage}
\end{center}
\caption{Distributions of the achieved \ref{eq:sl}-stationarity (left) and objective values (right)
of the final iterate $v^f$ for different values of $\alpha$ over randomly drawn initialization
points of our SLIP implementation. \ref{eq:sl}-stationarity is measured as the
$\ell_2$-norm of $(\nabla F(v^f)(t_i))_i$ for the switching locations
$t_i$ of $v^f$.}\label{fig:sensitivity_sl_obj}
\end{figure}

\begin{figure}
\begin{center}
  \includegraphics[height=4cm]{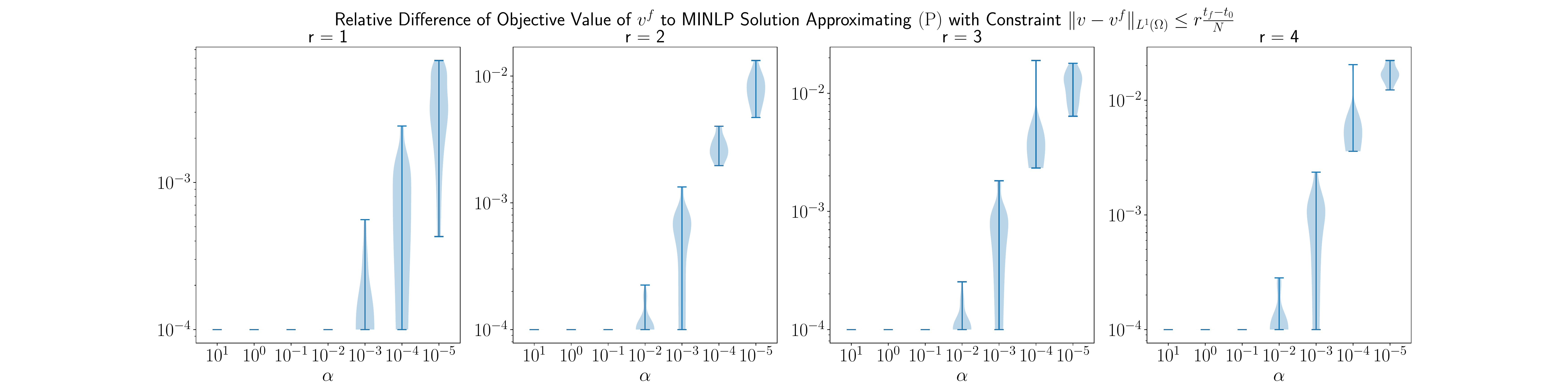}
\end{center}
\caption{Distributions of the achieved relative differences between the objective values achieved for $v^f$
and the MINLP solution with the additional constraint $\|v - v^f\|_{L^1(\Omega)} \le r\frac{t_f - t_0}{N}$
for $r \in \{1,2,3,4\}$. Values lower than $10^{-4}$ are set to $10^{-4}$ because this is the configured
duality gap threshold for the MINLP solver.}
\label{fig:sensitivity_gap}
\end{figure}

\section{Conclusion}\label{sec:conclusion}

Total variation regularization of integer optimal control problems
yields existence of minimizers under mild assumptions that
are independent of the regularity of the underlying dynamical system.
The set of feasible controls of the regularized problems is
closed with respect to weak-$^*$ and strict convergence
in $BV(\Omega)$. The infinite-dimensional vantage point allows for
a meaningful concept of local optimizers ($r$-optimal points)
for such problems.

We provide an algorithmic framework that combines a trust-region
globalization from nonlinear programming with IPs as subproblems to optimize for $r$-optimal
points. We obtain convergence of the resulting iterates
to points that satisfy a first-order necessary condition
of $r$-optimal points for problems with a one-dimensional domain
of the control input.
However, in order to obtain that \ref{eq:sl}-stationarity is a first-order
necessary condition for both the trust-region subproblem and
$r$-optimal points, a regularity assumption on the Hessian
of the objective $F$ and thus the underlying dynamical system
is necessary.

While the algorithmic framework is still computationally expensive,
we have experienced superior performance in terms of both
computational effort and achieved objective value when compared
with the application of an off-the-shelf spatial branch-and-bound
strategy to a discretized integer optimal control problem.
The results show that one can  easily achieve a
meaningful balance between the first term in the objective
(tracking in our example) and the total variation term. 
However, the algorithmic framework is computationally
less attractive than the combinatorial integral decomposition.
If fine control discretizations but low switching costs
are required or the first step of the combinatorial integral
approximation produces a control that is not close (in norm)
to an integer-valued one, the high-frequency switching
resulting from the combinatorial integral approximation
can be avoided by using the presented algorithmic framework.
The computational results, in particular Figure \ref{fig:control_trajectories},
show that although starting from the initial guess that is identical zero, the algorithm
is able to develop a nontrivial switching structure of the resulting control.
The computational results hint that, in particular for small
values of the regularization parameter $\alpha$, one may needs to choose fine discretizations,
thereby hinting that adaptive grid refinement strategies may be necessary for a sophisticated
implementation.

In order to save compute time in practice, the numerical results
from Experiment 2 suggest that
one should not  run the trust-region reset strategy
for fine grids directly on the finest required
grid but run it on refined meshes
and initialize it with the solution from the previous run.
Alternatively or additionally, one may  replace the trust-region 
reset strategy with a simple update strategy, for which
we, however, fail to guarantee the convergence properties
obtained in section \ref{sec:trm_analysis_1d}.
Contrary to our intuition the results indicate that a combination
of the trust-region update strategy with initialization from
the previous grid is not the fastest strategy in general.
Moreover we intend to use L-stationarity directly
in the algorithm in future implementations. This means that we
need to ensure more regularity of the problem \eqref{eq:sl}, in particular
differentiability of the reduced objective with respect to the
switching locations, and then solve it after the switching structure
has settled as is proposed in Remark \ref{rem:active_set_idea}.

Due to its similarity to PDE-constrained optimal
control problems with bound constraints, the algorithm might be
improved by integrating curvature information into the
trust-region subproblems, which may lead to more difficult
subproblems, however.

\section*{Acknowledgments}
The authors thank two anonymous referees for their helpful comments. The authors are also grateful to Gerd
Wachsmuth (BTU Cottbus-Senftenberg) for giving helpful
feedback on the manuscript.
This work was supported by the U.S.\ Department of Energy, Office of Science, Office of Advanced Scientific Computing Research, Scientific Discovery through the
Advanced Computing (SciDAC) Program through the FASTMath Institute under Contract No. DE-AC02-06CH11357.

\appendix

\section{Auxiliary Results}

We state the version of Taylor's theorem that we require for our analysis.
\begin{proposition}\label{prp:lin_approx}
Let $F : L^2(\Omega) \to \R$ be twice continuously
Fr\'{e}chet differentiable.
Then it holds for all $u$, $v \in L^2(\Omega)$ that
$F(v) = F(u) + (\nabla F(u), v - u)_{L^2(\Omega)}
+ \frac{1}{2}(v - u,\nabla^2F(\xi) (v - u))_{L^2(\Omega)}$
for some $\xi$ in the line segment between $v$ and $u$.
\end{proposition}
\begin{proof}
This follows from combining Taylor's theorem in Banach spaces (see section 4.5 in
\cite{zeidler2012applied}) with the intermediate value theorem, which
is applicable because $F$ is $\R$-valued.
\end{proof}

We state and prove the relationship between $\|\cdot\|_{L^2(\Omega)}$ and
$\|\cdot\|_{L^1(\Omega)}$ for $\{\nu_1,\ldots,\nu_M\}$-valued controls.
\begin{proposition}\label{prp:discretev_L1_L2_inequality}
Let $v \in L^\infty(\Omega)$ be $\{\nu_1,\ldots,\nu_M\}$-valued.
Then, $\|v\|_{L^2(\Omega)} \le M_1 \sqrt{\|v\|_{L^1(\Omega)}}$, where
$M_1 \coloneqq \max_{i} |\nu_i|$. Let
$v_1$, $v_2 \in L^\infty(\Omega)$ be $\{\nu_1,\ldots,\nu_M\}$-valued.
Then, $\|v_1 - v_2\|_{L^2(\Omega)} \le M_2 \sqrt{\|v_1 - v_2\|_{L^1(\Omega)}}$, where
$M_2 \coloneqq \max_{i,j} |\nu_i - \nu_i|$.
\end{proposition}
\begin{proof}
The first claim follows from the inequality
$\left(\int_\Omega |v(s)|^2\,\dd s\right)^{\frac{1}{2}}
\le \left(\int_\Omega \max_i|\nu_i||v(s)|\,\dd s\right)^{\frac{1}{2}}$.
The second claim follows analogously.
\end{proof}

We state and prove the regularity of $\nabla F(v)$ for
our computational example, in which we use a convolution
kernel of $W^{1,\infty}(\R)$-regularity.

\begin{proposition}\label{prp:derivative_of_example}
In the setting of section \ref{sec:experiments}
it holds for all $v \in L^2((t_0,t_f))$ that
\[ \nabla F(v) = K^* K v - K^* f \in C([t_0,t_f]). \]
\end{proposition}
\begin{proof}
The formula $\nabla F(v) = K^* K v - K^* f$ follows from basic
derivative calculus.
Because $k \in W^{1,\infty}_c
(\R)$, it holds that $Kv = (k * v)\chi_{(t_0,t_f)}$ is a continuous function.
Moreover, the adjoint of the convolution is the cross-correlation,
which is a convolution with the kernel $s \mapsto k(-x)$ instead of $k$.
This implies that $K^* K v$, $K^* f$, and in turn
$\nabla F(v)$ are continuous functions.
\end{proof}

\bibliography{biblio}
\bibliographystyle{plain}

\bigskip
\framebox{\parbox{.92\linewidth}{The submitted manuscript has been created by
UChicago Argonne, LLC, Operator of Argonne National Laboratory (``Argonne'').
Argonne, a U.S.\ Department of Energy Office of Science laboratory, is operated
under Contract No.\ DE-AC02-06CH11357.  The U.S.\ Government retains for itself,
and others acting on its behalf, a paid-up nonexclusive, irrevocable worldwide
license in said article to reproduce, prepare derivative works, distribute
copies to the public, and perform publicly and display publicly, by or on
behalf of the Government.  The Department of Energy will provide public access
to these results of federally sponsored research in accordance with the DOE
Public Access Plan \url{http://energy.gov/downloads/doe-public-access-plan}.}}
\end{document}